\documentclass[11pt,a4paper,oneside]{amsart}
\usepackage[latin1]{inputenc}
\usepackage{amsmath, amsfonts, amssymb,amsthm}

\author{Javier Cilleruelo 
	}\author{ Florian Luca}\author{Lewis Baxter}

\address{
Instituto de Ciencias Matem\'{a}ticas (CSIC-UAM-UC3M-UCM) and\newline
Departamento de Matem\'{a}ticas, Universidad Aut\'onoma de Madrid\newline
28049, Madrid, Espa\~na}
\email{franciscojavier.cilleruelo@uam.es} 

\address{
School of Mathematics, University of the Witwatersrand \newline
Private Bag X3, Wits 2050, South Africa and\newline
Max Planck Institute for Mathematics\newline
Vivatsgasse 7, 53111 Bonn, Germany 
}
\email{Florian.Luca@wits.ac.za}

\address{Sheridan College, School of Applied Computing,
1430 Trafalgar Road, Oakville, Ontario L6H 2L1}
\email{lewis.baxter@sheridancollege.ca}

\title{Every positive integer is a sum of three palindromes}
\date{\today}

\usepackage{todo}               

\newtheorem{thm}{Theorem}[section]
\newtheorem{theorem}[thm]{Theorem}

\newtheorem{lemma}[thm]{Lemma}

\newtheorem{prop}[thm]{Proposition}

\theoremstyle{definition}
\newtheorem{definition}[thm]{Definition}
\theoremstyle{remark}

\numberwithin{equation}{section}

\usepackage[a4paper]{geometry}
\begin{document}

\begin{abstract}
For integer $g\ge 5$, we prove that any positive integer can be written as a sum of three palindromes in base $g$.
\end{abstract}

\maketitle

\section{Introduction}
Let $g\ge 2$ be a positive integer. Any nonnegative integer $n$ has a unique base $g$ representation namely 
$$
n=\sum_{j\ge 0}\delta_jg^j,\quad {\text{\rm with}}\quad 0\le \delta_j\le g-1.
$$
The numbers $\delta_i$ are called the {\it digits of $n$ in base $g$}.  If $l$ is the number of digits of $n$, we use the notation
\begin{equation}
\label{eq:1}
n={\delta_{l-1}\cdots \delta_0},
\end{equation}
where we assume that $\delta_{l-1}\ne 0$. 

\begin{definition}
We say that $n$ is a base $g$ palindrome whenever $\delta_{l-i}=\delta_{i-1}$ holds for  all $i=1,\dots ,m=\lfloor l/2\rfloor$.
\end{definition}

There are many problems and results concerning the arithmetic properties of base $g$ palindromes. For example, in \cite{BHS} it is shown that almost all  base $g$ palindromes are composite. In \cite{BS}, it is shown that for every large $L$, there exist 
base $g$ palindromes $n$ with exactly $L$ digits and many prime factors (at least $(\log\log n)^{1+o(1)}$ of them as $L\to\infty$). The average value of the Euler function over binary (that is, with $g=2$) palindromes  $n$ with a fixed even number of digits 
was investigated in \cite{BS1}. In \cite{CLT} (see also \cite{L1}), it is shown that the set of numbers $n$ for which $F_n$, the $n$th Fibonacci number, is a base $g$ palindrome has asymptotic density zero as a subset of all positive integers, while in \cite{CLS} it was shown that base $g$ palindromes which are perfect powers (of some integer exponent $k\ge 2$) form a thin set as a subset of all base $g$ palindromes. In \cite{LT}, the authors found all positive integers $n$ such that $10^n\pm 1$ is a base $2$ palindrome, a result which was extended in \cite{BZ}. 
 
Recently, Banks \cite{Ba} started the investigation of the additive theory of palindromes by proving that every positive integer can be written as a sum of at most $49$ base $10$ palindromes. A natural question to ask would be how optimal is  the number $49$ in the above result. In this respect, we prove the following result.

 \begin{theorem}\label{main} 
 Let $g\ge 5$. Then any positive integer can be written as a sum of three base $g$ palindromes. 
 \end{theorem}
 
 The case $g=10$ of Theorem \ref{main} is a folklore conjecture which has been around for some time \cite{Friedman, Hass}. The paper \cite{Friedman} attributes a stronger conjecture to John Hoffman, namely that every positive integer $n$ can be written in base $g=10$ as a sum of three palindromes where one of them is the maximal palindrome less than or equal to $n$ itself. This was refuted in \cite{Sigg} which provided infinitely many examples of positive integers $n$ which are not a sum of two decimal palindromes.  
 
However, we prove that ``many" positive integers  are a sum of two palindromes.

\begin{theorem}\label{pp}
Let $g\ge 2$. There exists a positive constant $c_1$ depending on $g$ such that	
$$
|\{n\le x:\ n=p_1+p_2,~ p_1~\text{and}~p_2 ~\text{are base} ~g ~{\text palindromes}\}|\ge x^{1-\frac{c_1}{\sqrt{\log x}}}
$$
for all $x\ge 2$.
\end{theorem}

On the other hand the set of integers which are not the sum of two palindromes has positive density.
\begin{theorem}\label{ppu}For any $g\ge 3$ there exists a constant $c<1$ such that 
	$$|\{n\le x: n=p_1+p_2,\ p_1~\text{and}~p_2 ~\text{are base} ~g~ \text{palindromes }\}|\le cx$$ for $x$ large enough.
\end{theorem}

We do not know whether the set of positive integers which are the sum of two base $g$ palindromes has positive density. 

It would be interesting to extend Theorem \ref{main} to the missing bases $g\in \{2,3,4\}$. For $g=2$ we need at least four summands. It can be checked, for example, that $10110000$ is not a sum of two palindromes and it cannot be a sum of three palindromes either because it is an even number. For $g=3$ and $g=4$ we believe that some variant of our algorithms can show that three summands suffice.  Throughout this paper, we use the Landau symbols $O$ and $o$ as well as the Vinogradov symbols $\ll$ and $\gg$ with their usual meaning. These are used only in the proof of Theorem \ref{pp}. 
	
\section{The algorithms}
The proof of Theorem \ref{main} is algorithmic. That is, one can program the following proof to input a positive integer $n$ and obtain a representation of $n$ as a sum of three palindromes in base $g\ge 5$.
We assume throughout the proof that $g\ge 5$. 
 
For ease of notation, and using a convention introduced by Banks \cite{Ba}, we consider that $0$ is a base $g$ palindrome as well. For any integer $a$, we write $D(a)$ for that unique $d\in \{0,\dots,g-1\}$ such that $d\equiv a\pmod g$.

As in \eqref{eq:1}, we write the base $g$ representation of $n$ as
$$
n=\delta_{l-1}\dots \dots \dots \delta_1\delta_0.
$$
As before, $\delta_{l-1}\ne 0$. 
\subsection{Small cases} To present a clear algorithm, those integers with less than $7$ digits are considered separately in Section \S \ref{small}.

So, the algorithm starts by counting the number of digits of $n$. If $n$ has less than $7$ digits, then  Proposition \ref{smallcases} from Section \S \ref{small} shows how to write $n$ as a sum of three palindromes. If $n$ has  $7$ or more digits then we apply the general algorithm that we present in the next pages.

\subsection{The starting point}\label{initial}
For those integers with at least 7 digits, the starting point consists in assigning a {\it type} to $n$  according to the following classification. 
The type will define the lengths and the first digits (so, also the last) of the three palindromes $p_1,p_2,p_3$ that we will use to represent $n$. In the tables throughout the paper `$*$' denotes a known digit and `$.$' denotes a digit yet to be determined. 

\

{\bf Type A:  }

\

\begin{itemize}
	\item[A.1)] $\delta_{l-2}\ne 0,1,2,\quad z_1=D(\delta_0-\delta_{l-1}-\delta_{l-2}+1)\ne 0$.
	$$\begin{array}{|c|ccccccccccccccc|}\hline n&\delta_{l-1} &\delta_{l-2} & * & * & * & * &*& * &* & * &*& * & * & * & \delta_{0}\\  \hline p_1&\delta_{l-1} & . & . & .&. & . &. &. & . &. &. &.  &.  & .& \delta_{l-1}\\ p_2& & \delta_{l-2}-1 & . & . & . & . & . & . & . & . & . & . & . & . & \delta_{l-2}-1\\ p_3 & & & z_1 & . & . & . & . & . & . & . & . & . & .& .& z_1\\ \hline \end{array}$$ 
	
	\
	
	\item[A.2)] $ \delta_{l-2}\ne 0,1,2,\quad D(\delta_0-\delta_{l-1}-\delta_{l-2}+1)= 0$.
	$$\begin{array}{|c|ccccccccccccccc|}\hline n&\delta_{l-1} & \delta_{l-2} & * & * & * & * &*& * &* & * &*& * & * & * & \delta_{0}\\  \hline p_1 &\delta_{l-1} & . & . & .&. & . &. &. & . &. &. &.  &.  & .& \delta_{l-1}\\ p_2 & & \delta_{l-2}-2 & . & . & . & . & . & . & . & . & . & . & . & . & \delta_{l-2}-2\\ p_3& & & 1 & . & . & . & . & . & . & . & . & . & .& .& 1 \\ \hline\end{array}$$ 
	
	\
	
	\item[A.3)] $\delta_{l-2}= 0,1,2,\quad \delta_{l-1}\ne 1,\quad z_1=D(\delta_0-\delta_{l-1}+2)\ne 0$.
	$$\begin{array}{|c|ccccccccccccccc|}\hline n&\delta_{l-1} & \delta_{l-2} & * & * & * & * &*& * &* & * &*& * & * & * & \delta_{0}\\  \hline p_1&\delta_{l-1}-1 & . & . & .&. & . &. &. & . &. &. &.  &.  & .& \delta_{l-1}-1\\  p_2&&g-1 & . & . & . & . & . & . & . & . & . & . & . & . & g-1\\  p_3&& & z_1 & . & . & . & . & . & . & . & . & . & .& .& z_1 \\ \hline\end{array}$$ 
	
	\
	
	\item[A.4)] $\delta_{l-2}= 0,1,2\quad \delta_{l-1}\ne 1,\quad D(\delta_0-\delta_{l-1}+2)=0$.
	$$\begin{array}{|c|ccccccccccccccc|}\hline n&\delta_{l-1} & \delta_{l-2} & * & * & * & * &*& * &* & * &*& * & * & * & \delta_{0}\\  \hline p_1& \delta_{l-1}-1 & . & . & .&. & . &. &. & . &. &. &.  &.  & .& \delta_{l-1}-1\\  p_2& &g-2 & . & . & . & . & . & . & . & . & . & . & . & . & g-2\\  p_3& & & 1 & . & . & . & . & . & . & . & . & . & .& .& 1 \\ \hline\end{array}$$

	\item[A.5)] $ \delta_{l-1}=1,\quad \delta_{l-2}=0,\ \delta_{l-3}\le 3,\quad z_1=D(\delta_0-\delta_{l-3})\ne 0$.
	$$\begin{array}{|c|ccccccccccccccc|}\hline n&1 & 0 & * & * & * & * &*& * &* & * &*& * & * & * & \delta_{0}\\  \hline p_1&&g-1  & .&. & . &. &. & . &. &. &.  &.  & .&.&g-1\\ p_2& & & \delta_{l-3}+1 & . & . & . & . & . & . & . & . & . & . & . & \delta_{l-3}+1\\ p_3& & &  & z_1 & . & . & . & . & . & . & . & . & .& .& z_1 \\ \hline\end{array}$$ 
	
	\
	
	\item[A.6)] $ \delta_{l-1}=1,\quad \delta_{l-2}=0,\ \delta_{l-3}\le2,\quad D(\delta_0-\delta_{l-3})=0$.
	$$\begin{array}{|c|ccccccccccccccc|}\hline n&1 & 0 & * & * & * & * &*& * &* & * &*& * & * & * & \delta_{0}\\  \hline  p_1&&g-1  & .&. & . &. &. & . &. &. &.  &.  & .&.&g-1\\  p_2&& & \delta_{l-3}+2 & . & . & . & . & . & . & . & . & . & . & . & \delta_{l-3}+2\\  p_3&& &  & g-1 & . & . & . & . & . & . & . & . & .& .&g-1 \\ \hline\end{array}$$
		 
\end{itemize}

\newpage
		
{\bf Type B:} 

\medskip

\begin{itemize}
	\item[B.1)] $ \delta_{l-1}=1,\quad \delta_{l-2}\le 2,\quad \delta_{l-3}\ge  4,\quad z_1=D(\delta_0-\delta_{l-3})\ne 0$.
	$$\begin{array}{|c|ccccccccccccccc|}\hline n&1 & \delta_{l-2} & * & * & * & * &*& * &* & * &*& * & * & * & \delta_{0}\\  \hline p_1&1&\delta_{l-2} & . & .&. & . &. &. & . &. &. &.  &.  & \delta_{l-2}&1\\  p_2&& & \delta_{l-3}-1 & . & . & . & . & . & . & . & . & . & . & . & \delta_{l-3}-1 \\ p_3&& &  & z_1 & . & . & . & . & . & . & . & . & .& .&z_1\\ \hline \end{array}$$  
	
	\
	
	\item[B.2)] $ \delta_{l-1}=1,\quad \delta_{l-2}\le 2,\quad \delta_{l-3}\ge  3,\quad D(\delta_0-\delta_{l-3})= 0$.
	$$\begin{array}{|c|ccccccccccccccc|}\hline n&1 &\delta_{l-2}& * & * & * & * &*& * &* & * &*& * & * & * & \delta_{0}\\  \hline  p_1&1& \delta_{l-2} & . & .&. & . &. &. & . &. &. &.  &.  &\delta_{l-2}&1\\  p_2&& & \delta_{l-3}-2 & . & . & . & . & . & . & . & . & . & . & . & \delta_{l-3}-2 \\ p_3&& &  & 1 & . & . & . & . & . & . & . & . & .& .&1\\ \hline\end{array}$$

	\
		
	\item[B.3)] $ \delta_{l-1}=1,\quad \delta_{l-2}=1,2,\ \ \delta_{l-3}=0,1,\quad \delta_0=0 $.
	$$\begin{array}{|c|ccccccccccccccc|}\hline n&1 & \delta_{l-2} & \delta_{l-3}& * & * & * &*& * &* & * &*& * & * & * & \delta_{0}\\  \hline  p_1&1&\delta_{l-2}-1 & . & .&. & . &. &. & . &. &. &.  &.  & \delta_{l-2}-1&1\\  p_2&& &g-2& . & . & . & . & . & . & . & . & . & . & . & g-2\\  p_3&& &  &1 & . & . & . & . & . & . & . & . & .& .&1 \\ \hline\end{array}$$ 
	
	\
	
					\item[B.4)] $ \delta_{l-1}=1,\quad \delta_{l-2}=1,2,\ \ \delta_{l-3}=2,3,\quad \delta_0=0 $.
				$$\begin{array}{|c|ccccccccccccccc|}\hline n&1 & \delta_{l-2} & \delta_{l-3}& * & * & * &*& * &* & * &*& * & * & * & \delta_{0}\\  \hline  p_1&1&\delta_{l-2} & . & .&. & . &. &. & . &. &. &.  &.  & \delta_{l-2}&1\\  p_2&& &1& . & . & . & . & . & . & . & . & . & . & . & 1\\  p_3&& &  &g-2 & . & . & . & . & . & . & . & . & .& .&g-2
				 \\ \hline\end{array}$$ 
				 			
	\	
	
	\item[B.5)] $ \delta_{l-1}=1,\quad \delta_{l-2}=1,2,\ \ \delta_{l-3}=0,1,2,\quad z_1=\delta_0\ne 0 $.
	$$\begin{array}{|c|ccccccccccccccc|}\hline n&1 &\delta_{l-2} & \delta_{l-3}& * & * & * &*& * &* & * &*& * & * & * & \delta_{0}\\  \hline  p_1&1&\delta_{l-2}-1 & . & .&. & . &. &. & . &. &. &.  &.  & \delta_{l-2}-1&1\\ p_2& & &g-1& . & . & . & . & . & . & . & . & . & . & . & g-1\\  p_3& & &  &z_1 & . & . & . & . & . & . & . & . & .& .&z_1\\ \hline \end{array}$$ 
	
	\
	
	\item[B.6)] $ \delta_{l-1}=1,\quad \delta_{l-2}=1,2,\ \ \delta_{l-3}=3,\quad z_1=D(\delta_0-3)\ne 0 $.
	$$\begin{array}{|c|ccccccccccccccc|}\hline n&1 & \delta_{l-2} & 3& * & * & * &*& * &* & * &*& * & * & * & \delta_{0}\\  \hline  p_1&1&\delta_{l-2} & . & .&. & . &. &. & . &. &. &.  &.  & \delta_{l-2}&1\\  p_2&& &2& . & . & . & . & . & . & . & . & . & . & . & 2\\  p_3&& &  &z_1 & . & . & . & . & . & . & . & . & .& .&z_1
	\\ \hline\end{array}$$ 
			
		\
		
		\item[B.7)] $ \delta_{l-1}=1,\quad \delta_{l-2}=1,2,\ \ \delta_{l-3}=3,\quad \delta_0=3 $.
		$$\begin{array}{|c|ccccccccccccccc|}\hline n&1 & \delta_{l-2} & 3& * & * & * &*& * &* & * &*& * & * & * & \delta_{0}\\  \hline  p_1&1&\delta_{l-2} & . & .&. & . &. &. & . &. &. &.  &.  & \delta_{l-2}&1\\  p_2&& &1& . & . & . & . & . & . & . & . & . & . & . & 1\\  p_3&& &  &1 & . & . & . & . & . & . & . & . & .& .&1
		\\ \hline\end{array}$$ 	
	
		\end{itemize}

\smallskip

Notice that all the digits appearing in the classification are valid digits; i.e. $0\le \delta \le g-1$. We observe also that when $n$ if of type B, the digit of $p_1$ below $\delta_{l-3}$, which will be denoted by $x_2$, takes the values $0,1,2$ or $3$.

\subsection{ The algorithms} Once we have assigned the type to $n$ we have to check if $n$ is a {\it special number} or not.

\begin{definition}We say that $n$ is a {\it special number} if the palindrome $p_1$ corresponding to $n$ according the classification in types above has an even number of digits, say $l=2m$, and at least one of the digits $\delta_{m-1}$ or $\delta_m$ is equal to $0$. Otherwise we say that $n$ is a {\it normal number}.
\end{definition}

We use five distinct algorithms. We use Algorithms I, II, III and IV  for {\it normal numbers} and Algorithm V  for {\it special numbers}.

\begin{itemize}	
\item[Algorithm I:] To be applied to integers such that the associated palindromes $p_1,p_2,p_3$ have $2m+1,\ 2m,\ 2m-1$ digits respectively for some $m\ge 3$. In other words, those of type A1, A2, A3 and A4 when $l=2m+1$ and those of type A5 and A6 when $l=2m+2$. The cases $m\le 2$ correspond to the {\it small cases}.

\smallskip

\item[Algorithm II:]  To be  applied to integers such that the associated palindromes $p_1,p_2,p_3$ have $2m,\ 2m-1,\ 2m-2$ digits respectively for some $m\ge 3$ and such that  $\delta_{m-1}\ne 0$ and $ \delta_m\ne 0$. In other words, those of type A1, A2, A3 and A4 when $l=2m$ and $\delta_{m-1}\ne 0$ and $ \delta_m\ne0$ and those of type A5 and A6 when $l=2m+1$ and  $\delta_{m-1}\ne 0$ and $ \delta_m\ne 0$.  The cases $m\le 2$ correspond to the {\it small cases}.

\smallskip

\item[Algorithm III:]  To be applied to integers such that the associated palindromes $p_1,p_2,p_3$  have $2m+1,2m-1,2m-2$ digits respectively for some $m\ge 3$.  In other words, those of type B with $l=2m+1$. The cases $m\le 2$ correspond to the {\it small cases}.

\smallskip

\item[Algorithm IV:] To be applied to integers such that the associated palindromes $p_1,p_2,p_3$  have $2m,2m-2,2m-3$ digits respectively for some $m\ge 4$. In other words, those of type B with $l=2m$  and with $\delta_m\ne 0$ and $ \delta_{m-1}\ne 0$. The cases $m\le 3$ correspond to the {\it small cases}.

\smallskip

\item[Algorithm V:] To be applied to {\it special numbers} that are not covered by the {\it small cases}.
\end{itemize}

\subsection{Algorithm I} 

 Assume $m\ge 3$. The {\it initial configuration} when we apply Algorithm I is one of the  following configurations:
	$$\begin{array}{|ccccccccccccccc|}\hline \delta_{2m} & \delta_{2m-1} & \delta_{2m-2} & * & * & * &*& * &* & * &*& * & * & \delta_1 & \delta_0\\  \hline x_1 & . & . & .& . & . &. &. & . &. & . &. &.  &. & x_1\\  & y_1 & . & . & . & . & . & . & . & . & . & . & . & . & y_1\\  & & z_1 & . & . & . & . & . & . & . & . & . & .& .& z_1 \\ \hline\end{array}$$
	
		$$\begin{array}{|cccccccccccccccc|}\hline 1&\delta_{2m} & \delta_{2m-1} & \delta_{2m-2} & * & * & * &*& * &* & * &*& * & * & \delta_1 & \delta_0\\  \hline &x_1 & . & . & .& . & . &. &. & . &. & . &. &.  &. & x_1\\ & & y_1 & . & . & . & . & . & . & . & . & . & . & . & . & y_1\\ & & & z_1 & . & . & . & . & . & . & . & . & . & .& .& z_1 \\ \hline\end{array}$$
		
 Algorithm I in either case is the following:
	
	\smallskip
	
{\bf Step 1:} We choose $x_1,y_1,z_1$ according to the configurations described in the starting point. Define $c_1=(x_1+y_1+z_1)/g,$ which is the carry of the column 1.

\smallskip

{\bf Step $2$:} Define the digits
\begin{eqnarray*}  	
	x_2&=&
	\left \{\begin{matrix} D(\delta_{2m-1}-y_1) & \text{if } & \quad z_{1}\le \delta_{2m-2}-1;\\	D(\delta_{2m-1}-y_1-1)   & \text{ if } & z_{1}\ge \delta_{2m-2};	\end{matrix}\right .
	\\
	y_2&=& D( \delta_{2m-2}-z_{1}-1 ); \\
	z_2&=&D(\delta_{1}-x_2-y_{2}-c_{1});\\ c_2&=&(x_2+y_2+z_2+c_{1}-\delta_{1})/g\quad  (\text{the carry from column}\quad  2).
\end{eqnarray*}	
	
{\bf Step $i,\ 3\le i\le m$:}  Define the digits 
 \begin{eqnarray*}  
 	x_i&=&\left\{\begin{matrix} 1 & \text{if } & \quad z_{i-1}\le \delta_{2m-i}-1;\\ 
 		0  & \text{if} & z_{i-1}\ge \delta_{2m-i};
 	\end{matrix} 
 	\right.\\
 y_i&=&D( \delta_{2m-i}-z_{i-1}-1  ); \\
 z_i&=&D(\delta_{i-1}-x_i-y_{i}-c_{i-1});\\ c_i&=&(x_i+y_i+z_i+c_{i-1}-\delta_{i-1})/g\quad  (\text{the carry from column}\quad  i).
   \end{eqnarray*}
   
   {\bf Step $ m+1$:} Define
   $$x_{m+1}=0.   $$

The diagram below represents the configuration after step $i$:
 $$\begin{array}{|cccccccccccccccc|}\hline \dots &\delta_{2m-i+1} & \delta_{2m-i} & \delta_{2m-i-1} & * & * & * &*& * &* & * &*& * & * & \delta_{i-1} &\dots\\  \hline  \dots& x_i& . & . & .& . &. & . &. & .&.& . &.  &.  & x_i & \dots \\ \dots &  y_{i-1}&  y_i & . & . & . & . & . & . & . & . & . & . & . & y_i  &\dots\\  \dots & z_{i-2}&z_{i-1} & z_i&  . & . & . & . & . & . & . & . & . & .& z_{i}&\dots \\ \hline\end{array}$$ 
 
  A few words to explain what is behind the algorithm:
  
 The digit $y_i$ is defined to adjust the digit $\delta_{2m-i}$ from the left side once we know the digit $z_{i-1}$ and assuming a possible carry  from the previous column (the $-1$ in the definition of $y_i$ takes into account this possible carry).  The $z_i$ is defined to adjust the digit $\delta_{i-1}$ in the right side once we know $x_i,y_i$ and $c_{i-1}$, the carry from the previous column. Now we go again to the left side.  If $z_{i}\ge \delta_{2m-i-1}$ we will get the possible carry we had assumed and then we define $x_{i+1}=0$. If $z_{i}\le \delta_{2m-i-1}-1$ we do not get any carry and then we define $x_{i+1}=1$, which has the same effect as the carry that we had expected.

After the last step
 the configuration that we obtain is the following:
  $$\begin{array}{|ccccccc|ccccccccc|}\hline \delta_{2m} & \delta_{2m-1} & \delta_{2m-2} & * & * & * & \delta_{m} & \delta_{m-1}&\delta_{m-2} & *& * &*& * & * & \delta_1 & \delta_0\\  \hline x_1 & * & *& * & *& x_m& 0 &x_m &* & * &* & * &*  &* & * & x_1\\  & y_1 & * & *&* & y_{m-1} & y_m&  y_m &y_{m-1}&*& *&*& *& * & * & y_1\\  & & z_1 & * & * & * & z_{m-1} & z_m&z_{m-1}& * & * & * & * & *& *& z_1\\ \hline \end{array}$$
  
  We call {\it temporary configuration} the configuration we get after the last step. We have drawn a vertical line where  both sides of the algorithm collide. It is not true in general that $n$ is equal to the sum of the three palindromes we obtain in the temporary configuration. 
  
  If $\Delta_m$ is the digit we obtain in column $m+1$ when we sum the three palindromes, we observe that
  $$\Delta_m\equiv y_m+z_{m-1}+c_m\equiv \delta_m+c_m-1\pmod g.$$
  
  If $c_m=1$ then $\Delta_m=\delta_m$ and we obtain the correct digit in column $m+1$ and, as consequence of Proposition \ref{induction}, we obtain  the correct digit in all the columns. In this case $n$ is equal to the sum of the three palindromes of the temporary configuration so the {\it temporary configuration} is also the {\it final configuration}.
  
  If $c_m\ne 1$, then we need an extra adjustment. 
  
  \subsection{The adjustment step}
 
 For $i=0,\dots, 2m,$ we denote by $\Delta_i$ the digit we obtain in column $i+1$ when we sum the three palindromes that we have obtained after the last step. 
  Of course we want that $\Delta_i=\delta_{i}$ for all $i,\ 0\le i\le 2m$. Unfortunately, this is not always true but it is almost true. The following proposition shows that we obtain the correct  digits on the left side (thanks to the $z_i$'s) and that we obtain the correct digit in a column of the right side if the digit we obtain in the previous column is also the correct digit.
 \begin{prop}\label{induction}Let $g\ge 5$ and $m\ge 3
 	$. We have that $\Delta_i=\delta_{i}$ for all $0\le i\le m-1$. Furthermore, for any $0\le i\le m-1$, if $\Delta_{m+i}=\delta_{m+i}$, then $\Delta_{m+i+1}=\delta_{m+i+1}$.
 \end{prop}
 \begin{proof}The first statement of the proposition is clear because of the way we have defined the $z_i$'s. As for the second statement, we prove it separately for $i=0$, for $1\le i\le m-3$, for $i=m-2$ and for $i=m-1$.
 	
 	\begin{itemize}
 		\item [i)] $i=0$. We have
 		 \begin{eqnarray*} \Delta_{m+1}&\equiv & x_{m}+y_{m-1}+z_{m-2}+c_{m+1}\\ &\equiv & \delta_{m+1}+x_{m}+c_{m+1}-1\pmod g.\end{eqnarray*} Then we have to prove that $x_{m}+c_{m+1}=1$. 
 		 \begin{itemize}
 		 	
 		 	\item[a)] If $x_{m}=1$ then $z_{m-1}\le \delta_{m}-1$, so $y_{m}=\delta_{m}-z_{m-1}-1$. 
 		 	Since \begin{eqnarray*}
 		 		\Delta_{m}&\equiv & y_{m}+z_{m-1}+c_{m}\equiv \delta_{m}+c_{m}-1 \pmod g,\end{eqnarray*} and we have assumed that $\Delta_{m}= \delta_{m}$, we conclude that $c_{m}\equiv 1\pmod g$, so $c_{m}=1$ (because $|c_{m}-1  |\le 2<g$). Thus,
 		 	\begin{eqnarray*}c_{m+1}&=&(y_{m}+z_{m-1}+c_{m}-\delta_{m})/g=(c_{m}-1)/g=0,
 		 	\end{eqnarray*}
 		 	and then $x_m+c_{m+1}=1$.
 		 	
 		 	\item[b)] If $x_{m}=0$, then $z_{m-1}\ge \delta_{m}$, so $y_{m}=g+\delta_{m}-z_{m-1}-1$. The argument is similar to the one before except that now we get
 		 	\begin{eqnarray*}c_{m+1}&=&(y_{m}+z_{m-1}+c_{m}-\delta_{m})/g=(g+c_{m}-1)/g=1,
 		 	\end{eqnarray*}   	 		
 		 	and again  $x_m+c_{m+1}=1$.
 		 \end{itemize}
 		 In any case, we have that $x_{m}+c_{m+1}=1$, and then $\Delta_{m+1}= \delta_{m+1}.$
 		 
 		 \smallskip
 		
 		\item[ii)] $1\le i\le m-3$ (these cases are vacuous for $m=3$):
 \begin{eqnarray*} \Delta_{m+i+1}&\equiv & x_{m-i}+y_{m-i-1}+z_{m-i-1-2}+c_{m+i+1}\\ &\equiv & \delta_{m+i+1}+x_{m-i}+c_{m+i+1}-1\pmod g.\end{eqnarray*} We have to prove that $x_{m-i}+c_{m+i+1}=1$. 
 	\begin{itemize}
 	
 \item[a)] If $x_{m-i}=1$, then $z_{m-i-1}\le \delta_{m+i}-1$, so $y_{m-i}=\delta_{m+i}-z_{m-i-1}-1$.   
 Since \begin{eqnarray*}
  \Delta_{m+i}&\equiv & x_{m-i+1}+y_{m-i}+z_{m-i-1}+c_{m+i}\\ &\equiv &x_{m-i+1}+\delta_{m+i}-1+c_{m+i} \pmod g,\end{eqnarray*} and we have assumed that $\Delta_{m+i}=\delta_{m+i}$, we conclude that 
  $$
  x_{m-i+1}+c_{m+i}-1\equiv 0\pmod g,
  $$ 
  therefore $x_{m-i+1}+c_{m+i}-1=0$ (because $|x_{m-i+1}+c_{m+i}-1  |\le 2$). Thus,
    \begin{eqnarray*}c_{m+i+1}&=&(x_{m-i+1}+y_{m-i}+z_{m-i-1}+c_{m+i}-\delta_{m+i})/g\\&=&(x_{m-i+1}-1+c_{m+i})/g=0,
   	\end{eqnarray*}
   	and $x_{m-i}+c_{m+i+1}=1$.
   	   	
   	\item[b)] If $x_{m-i}=0$, then $z_{m-i-1}\ge \delta_{m+i}$, so $y_{m-i}=g+\delta_{m+i}-z_{m-i-1}-1$. The argument is similar to one before except that now we get
    \begin{eqnarray*}c_{m+i+1}&=&(x_{m-i+1}+y_{m-i}+z_{m-i-1}+c_{m+i}-\delta_{m+i})/g\\&=&(g+x_{m-i+1}-1+c_{m+i})/g=1,
    \end{eqnarray*}   	and again  $x_{m-i}+c_{m+i+1}=1$.		
   		\end{itemize}
   		In any case, we have that $x_{m-i}+c_{m+i+1}=1$ and then $\Delta_{m+i+1}=\delta_{m+i+1}.$
   	\item[iii)] $i=m-2$. 
   	 We have
   		$$\Delta_{2m-1}\equiv x_2+y_1+c_{2m-1}\pmod g.$$
   		We distinguish two cases:
   		\begin{itemize}
   		\item[a)]If $z_1\le \delta_{2m-2}-1$, then $y_2=\delta_{2m-2}-z_1-1$ and
   			$$
			\Delta_{2m-1}\equiv x_2+y_1+c_{2m-1}\equiv \delta_{2m-1}+c_{2m-1}\pmod g.
			$$  
			Since
   			$$
			\Delta_{2m-2}\equiv x_3+y_2+z_1+c_{2m-2}\equiv x_3+\delta_{2m-2}-1+c_{2m-2}\pmod g,
			$$ 
			and we have assumed that $\Delta_{2m-2}= \delta_{2m-2}$, we get
			 $x_3-1+c_{2m-2}=0$ (because $|x_3-1+c_{2m-2}|\le 2$ ). Thus, 
			 $$
			 c_{2m-1}=(x_3+y_2+z_1+c_{2m-2}-\delta_{2m-2})/g= 0,
			 $$ 
			 and we have $\Delta_{2m-1}=\delta_{2m-1}.$
   			
   			\item[b)] If $z_1\ge \delta_{2m-2}$, then $y_2=g+\delta_{2m-2}-z_1-1$ and $$\Delta_{2m-1}\equiv x_2+y_1+c_{2m-1}\equiv \delta_{2m-1}+c_{2m-1}-1\pmod g.$$  We repeat the same argument as in case a) except that now	
   			$$c_{2m-1}=(x_3+y_2+z_1+c_{2m-2}-\delta_{2m-2})/g= 1,$$ and again $\Delta_{2m-1}=\delta_{2m-1}.$
   			   		\end{itemize}
   			
   			\item[iv)] $i=m-1$. We can check in the classification in types that if $\Delta_{2m-1}=\delta_{2m-1}$, then $\Delta_{2m}=\delta_{2m}.$  In other words, that we have $c_{2m}=0$ for the types	A1 and A2 and we have $c_{2m}=1$ for the types A3, A4, A5 and A6. 
	
   			\end{itemize}
 	\end{proof}

	Proposition \ref{induction} shows that if $\Delta_m=\delta_m$ then $\Delta_i=\delta_i$ for all  $i=0,\ldots,2m$ and then the three palindromes we have obtained do the job.

The problem appears when $\Delta_m\ne \delta_m$ and this occurs when $c_m\ne 1$. When this happens, we need to make an adjustment to our {\it temporary} configuration.

 Notice that for $m\ge 3$ we have $$\Delta_m\equiv  \delta_m+c_m-1\pmod g,$$   and that $c_m$ takes the value $0,1$ or $2$. 

All the possible situations are considered in the cases below:

{\bf I.1} $\quad \boldsymbol{c_m=1}$. In this case  $\Delta_m=\delta_m$ and  there is nothing to change. The temporary configuration is simply the final configuration since in all columns the sums of the digits including the carries yield the digits of $n$.  

{\bf I.2} $\quad \boldsymbol{c_m=0}$. In this case we need to increment by one unit the digit we obtain in the column $m+1$. We can do this by changing the value of $x_{m+1}=0$ to $x_{m+1}=1$.
 $$\begin{array}{|c|c|}   \hline\delta_{m} & \delta_{m-1} \\ \hline  0 &* \\   y_m&  y_m\\   * & z_m \\ \hline\end{array}\qquad  \longrightarrow \qquad \begin{array}{|c|c|} \hline \delta_{m} & \delta_{m-1} \\  \hline 1 &* \\   y_m&  y_m\\   * & z_m \\ \hline\end{array}$$
 Notice that we have modified the central digit of the first palindrome, so the new first row is also a palindrome. Notice also that now we obtain the correct digit in column $m+1$ and also in all remaining columns.

{\bf I.3} $\quad \boldsymbol{c_m=2}$. In this case, we have that $y_m\ne 0$ (otherwise $c_m\ne 2$). Further, if $z_m\ne g-1$, then the only possibility to have $c_m=2$ is that $z_m=g-2,~y_m=g-1,~x_m=1$ and $c_{m-1}=2$, but that
gives $\delta_{m-1}=0$, which is not allowed. Thus, $\boldsymbol{z_m= g-1}$ and we make the following adjustment:

  	  $$\begin{array}{|c|c|} \hline  \delta_{m} & \delta_{m-1} \\ \hline  0 &* \\   y_m&  y_m\\   *& g-1 \\ \hline\end{array}\qquad  \longrightarrow \qquad \begin{array}{|c|c|}  \hline\delta_{m} & \delta_{m-1} \\  \hline 1 &*\\   y_m-1&  y_m-1\\   * &0 \\ \hline\end{array}$$

  	 	  	 \
  	 
Observe that in every adjustment step we have been successful in   increasing or decreasing the digit that was obtained in the column $m+1$ when $c_m=0 \text{ or } 2,$ without altering the digits from the previous column. Notice also that in every adjustment we always modify the central digits of the temporary palindromes such that the new ones are also palindromes. Once we have realized these adjustments, 
the digit we get in the column $m+1$ is $\delta_m$, the correct digit, and Proposition \ref{induction} proves that all the digits are correct.

\subsection{The three palindromes and an example}

We end this subsection by illustrating the application of Algorithm I to an example. Let  $n$ be the positive integer giving the first $21$ decimal digits of $\pi$:
$$
n=314159265358979323846.
$$
We see that $n$ is of type A1, therefore the configuration after Step 1 is the following:
$$\begin{array}{|ccccccccccc|cccccccccc|}\hline 3 & 1 & 4 & 1 & 5 & 9 & 2& 6&5 & 3& 5 &8& 9 & 7 & 9 & 3&2&3&8& 4& 6\\  \hline  2 & . & . & . & . & . &.& .&. & .& . &.& . & . & .& .&.&.&.& .& 2\\  & 9 & . &. & . & . & .& .&. & .& . &.& . & . & . &.&.&.&.& .& 9\\ &  &5 &. & . & . & .& .&. & .& . &.& . & . & . &.&.&.&.& .& 5\\ \hline
\end{array}
$$
Thus $n$ is a normal integer and we can apply Algorithm I.

Since $z_1\ge \delta_{2m-2}$, Step 2 starts defining 
\begin{eqnarray*}x_2&=& D(\delta_{2m-1}-y_1-1)=D(1-9-1)=1,\\
y_2&=&D(\delta_{2m-2}-z_1-1)=D(4-5-1)=8,\\ 
z_2&=&D(\delta_1-x_2-y_2-c_1)=D(4-1-8-1)=4,\\
c_2&=&(x_2+y_2+z_2+c_1-\delta_1)/10=1,
\end{eqnarray*}
and the configuration after Step $2$ is
$$\begin{array}{|ccccccccccc|cccccccccc|}\hline 3 & 1 & 4 & 1 & 5 & 9 & 2& 6&5 & 3& 5 &8& 9 & 7 & 9 & 3&2&3&8& 4& 6\\  \hline  2 & \bf 1 & . & .& . & . &.& .&.& .& . &.& . & . & .& .&.&.&.& \bf 1& 2\\  & 9 & \bf 8 &. & . & . & .& .&. & .& . &.& . & . & . &.&.&.&.& \bf 8& 9\\ &  &5 & \bf 4 & . & . & .& .&. & .& . &.& . & . & . &.&.&.&.&\bf 4& 5\\ \hline
\end{array}$$
and after Step $3$ is $$\begin{array}{|ccccccccccc|cccccccccc|}\hline 3 & 1 & 4 & 1 & 5 & 9 & 2& 6&5 & 3& 5 &8& 9 & 7 & 9 & 3&2&3&8& 4& 6\\  \hline  2 &  1 & \bf 0 & .& . & . &.& .&.& .& . &.& . & . & .& .&.&.&\bf 0& 1& 2\\  & 9 &  8 &\bf 6 & . & . & .& .&. & .& . &.& . & . & . &.&.&.&\bf 6&  8& 9\\ &  &5 & 4 & \bf 1& . & .& .&. & .& . &.& . & . & . &.&.&.&\bf 1& 4& 5\\ \hline
\end{array}$$

Continuing with the algorithm we get to the temporary configuration:
$$\begin{array}{|ccccccccccc|cccccccccc|}\hline 3 & 1 & 4 & 1 & 5 & 9 & 2& 6&5 & 3& 5 &8& 9 & 7 & 9 & 3&2&3&8& 4& 6\\  \hline  2 & 1 & 0 &  1 & 0 & 0 &1& 0&0 & 1&  0 & 1& 0 & 0 &  1& 0&0&1&0& 1& 2\\  & 9 & 8 &6 & 3 &9 & 9& 2&9 & 4& 0 &0& 4 & 9 & 2 &9&9&3&6& 8& 9\\ &  &5 &4 & 1 & 9 & 2& 3&5 & 8& 4 &7& 4 & 8 & 5 &3&2&9&1&4& 5\\ \hline
\end{array}$$
Since $c_m=0$, we need to apply Adjustment I.2 and obtain the final configuration:
$$\begin{array}{|c|ccccccccccc|cccccccccc|}\hline n&3 & 1 & 4 & 1 & 5 & 9 & 2& 6&5 & 3& 5 &8& 9 & 7 & 9 & 3&2&3&8& 4& 6\\  \hline  p_1&2 & 1 & 0 &  1 & 0 & 0 &1& 0&0 & 1&  \bf 1 & 1& 0 & 0 &  1& 0&0&1&0& 1& 2\\ p_2& & 9 & 8 &6 & 3 &9 & 9& 2&9 & 4& 0 &0& 4 & 9 & 2 &9&9&3&6& 8& 9\\ p_3&&  &5 &4 & 1 & 9 & 2& 3&5 & 8& 4 &7& 4 & 8 & 5 &3&2&9&1&4& 5\\ \hline
\end{array}$$

\section{The remaining cases}

\subsection{Algorithm II}  The algorithm only differs in the subindices of the $\delta_i$'s (because now $l=2m$ is even) and in the adjustment step, which is slightly more complicated to describe because of the many cases to be considered. The cases $m\le 2$ correspond to the {\it small cases}. For $m\ge 3$, we proceed in the following steps:

\smallskip

{\bf Step 1:} We choose $x_1,y_1,z_1$ according to the configurations described in Section \ref{initial}. Define $c_1=(x_1+y_1+z_1)/g,$ which is the carry of the column 1.

\smallskip

{\bf Step $2$:} Define the digits
\begin{eqnarray*}  	
	x_2&=&
	\left \{\begin{matrix} D(\delta_{2m-2}-y_1) & \text{if } & \quad z_{1}\le \delta_{2m-3}-1;\\	D(\delta_{2m-2}-y_1-1)   & \text{ if } & z_{1}\ge \delta_{2m-3};	\end{matrix}\right .
	\\
	y_2&=& D( \delta_{2m-3}-z_{1}-1 ); \\
	z_2&=&D(\delta_{1}-x_2-y_{2}-c_{1});\\ c_2&=&(x_2+y_2+z_2+c_{1}-\delta_{1})/g\quad  (\text{the carry from column}\quad  2).
\end{eqnarray*}		
{\bf Step $i,\ 3\le i\le m-1$} (these steps are vacuous for $m=3$):  Define the digits
\begin{eqnarray*}  
	x_i&=&\left\{\begin{matrix} 1 & \text{if } & \quad z_{i-1}\le \delta_{2m-i-1}-1;\\ 
		0  & \text{if} & z_{i-1}\ge \delta_{2m-i-1};
	\end{matrix} 
	\right.\\
	y_i&=&D( \delta_{2m-i-1}-z_{i-1}-1  ); \\
	z_i&=&D(\delta_{i-1}-x_i-y_{i}-c_{i-1});\\ c_i&=&(x_i+y_i+z_i+c_{i-1}-\delta_{i-1})/g\quad  (\text{the carry from column}\quad  i).
\end{eqnarray*}
{\bf Step $m$:}  Define the digits
\begin{eqnarray*}  
	x_m&=&0.
	\\
	y_m&=&D( \delta_{m-1}-z_{m-1}-c_{m-1}  ).
\end{eqnarray*}

 The temporary configuration is:
$$\begin{array}{|ccccccc|ccccccccc|}\hline\delta_{2m-1} & \delta_{2m-2} & \delta_{2m-3} & * & * & * & \delta_{m} & \delta_{m-1}&\delta_{m-2} & *& * &*& * & * & \delta_1 & \delta_0\\  \hline x_1 & . & .& . & .& . & 0& 0 &x_{m-1} & . &.& . &.  &.  & . & x_1\\  & y_1 & . & . & . & . &y_{m-1}& y_{m}&y_{m-1} & .& .& . & . & . & . & y_1\\  & & z_1 & . & . & . & . & z_{m-1} & z_{m-1} & . & . & . & . & .& .& z_1 \\ \hline
\end{array}$$ or
$$\begin{array}{|cccccccc|ccccccccc|}\hline 1&\delta_{2m-1} & \delta_{2m-2} & \delta_{2m-3} & * & * & * & \delta_{m} & \delta_{m-1}&\delta_{m-2} & *& * &*& * & * & \delta_1 & \delta_0\\  \hline & x_1 & . & .& . & .& . & 0 & 0 &x_{m-1} & . &.& . &.  &.  & . & x_1\\  & & y_1 & . & . & . & . &y_{m-1}& y_{m}&y_{m-1} & .& .& . & . & . & . & y_1\\ & & & z_1 & . & . & . & . & z_{m-1} & z_{m-1} & . & . & . & . & .& .& z_1 \\ \hline
\end{array}$$
with $\delta_{m-1}\ne 0$ and $\delta_m\ne 0$.

 \begin{prop}Let $g\ge 5$ and $m\ge 3$. We have that $\Delta_i=\delta_{i}$ for all $0\le i\le m-1$. Furthermore, for any $0\le i\le m-2$, if $\Delta_{m+i}=\delta_{m+i}$, then $\Delta_{m+i+1}=\delta_{m+i+1}$.
 \end{prop}
 \begin{proof}The proof is similar to the proof of Proposition 
 \ref{induction}. We only give the details for  $i=0$, which is the only case somewhat different.
 
 Assume that $\Delta_{m}=\delta_m$. In other words, that $(y_{m-1}+z_{m-2}+c_m-\delta_m)/g$ is an integer. We have
 $$\Delta_{m+1}\equiv x_{m-1}+y_{m-2}+z_{m-3}+c_{m+1}\equiv x_{m-1}+\delta_{m+1}-1+c_{m+1}\pmod g.$$
 
 If $x_{m-1}=0$, then $z_{m-2}\ge \delta_m$ and $y_{m-1}=g+\delta_m-z_{m-2}-1$. Thus,
 $$c_{m+1}=(y_{m-1}+z_{m-2}+c_m-\delta_m)/g=(g+c_m-1)/g=1$$ because $c_{m+1}$ is an integer and $|c_m-1|\le 1<g.$
 
  If $x_{m-1}=1$, then $z_{m-2}\le \delta_m-1$ and $y_{m-1}=\delta_m-z_{m-2}-1$. Thus,
  $$c_{m+1}=(y_{m-1}+z_{m-2}+c_m-\delta_m)/g=(c_m-1)/g=0$$ because $c_{m+1}$ is an integer and $|c_m-1|\le 1<g.$
  
  In any case, we have that $x_{m-1}+c_{m+1}=1$, so $\Delta_{m+1}\equiv \delta_{m+1}.$
 \end{proof}
 
 The above proposition implies that if $\Delta_m=\delta_m$, then $\Delta_i=\delta_i$ for all $i=0,\ldots,2m-1$.
 
{\bf Adjustment step:} Notice that $\Delta_{m}\equiv\delta_{m}+c_m-1\pmod g$. Thus, we make the adjustment according to this observation.

{\bf II.1} $\quad \boldsymbol{c_{m}=1}$. We do  nothing and the temporary configuration becomes the final one.

\smallskip

{\bf II.2} $\quad \boldsymbol{c_{m}=0}$. We distinguish the following cases:
\begin{itemize} 
\item[II.2.i)] $\boldsymbol{y_{m}\ne 0}$. 
 $$\begin{array}{|c|c|} \hline \delta_{m}&\delta_{m-1}\\  \hline  0 & 0 \\  *&y_{m}  \\  *&*
\\ \hline\end{array} \qquad \longrightarrow \qquad \begin{array}{|c|c|} \hline \delta_{m}&\delta_{m-1}\\  \hline  1 & 1 \\  *&y_{m}-1  \\  *&*
\\ \hline\end{array} $$	
\item[II.2.ii)] $\boldsymbol{y_{m}=0}$. 
\begin{itemize}
	\item[II.2.ii.a)] $\boldsymbol{y_{m-1}\ne 0}$.
	
	 $$\begin{array}{|c|cc|} \hline\delta_{m} & \delta_{m-1}&\delta_{m-2}\\  \hline 0 & 0 &* \\   y_{m-1}& 0& y_{m-1}\\  * &z_{m-1} & z_{m-1}
	 \\ \hline\end{array}\quad \longrightarrow \quad \begin{array}{|c|cc|} \hline\delta_{m} & \delta_{m-1}&\delta_{m-2}\\  \hline 1 & 1 &* \\   y_{m-1}-1& g-2& y_{m-1}-1\\  * &z_{m-1}+1 & z_{m-1}+1
	 \\ \hline\end{array}$$
	 The above step is justified for $z_{m-1}\ne g-1$. But if $z_{m-1}=g-1$, then $c_{m-1}\ge (y_{m-1}+z_{m-1})/g\ge 1$, so $c_m=(z_{m-1}+c_{m-1})/g=(g-1+1)/g=1$, a contradiction. 
	 
	 		\item[II.2.ii.b)] $\boldsymbol{y_{m-1}= 0,\, z_{m-1}\ne 0}$.
	 		$$\begin{array}{|c|cc|} \hline\delta_{m} & \delta_{m-1}&\delta_{m-2}\\  \hline 0 & 0 &* \\   0& 0& 0\\  * &z_{m-1}& z_{m-1}
	 		\\ \hline\end{array}\quad \longrightarrow \quad \begin{array}{|c|cc|} \hline\delta_{m} & \delta_{m-1}&\delta_{m-2}\\  \hline 0 & 0 &* \\   1& 1& 1\\  * &z_{m-1}-1 & z_{m-1}-1
	 		\\ \hline\end{array}$$
		\item[II.2.ii.c)] $\boldsymbol{y_{m-1}=0,\, z_{m-1}=0}$. 	
			
			If also $c_{m-1}=0$, then $\delta_{m-1}=0$, which is not allowed. Thus, $c_{m-1}=1$. This means that $x_{m-1}\in \{g-1,g-2\}$. Since $x_i\in \{0,1,2\}$ for $i\ge 3$, it follows that $m=3$ and we are in one of the cases A.5) or A.6). Further, $\delta_2=1$.  In this case we change the above configuration to: 
			$$\begin{array}{|cc|cc|} \hline \delta_{m+1} & \delta_{m} & \delta_{m-1}&\delta_{m-2}\\  \hline x_{m-1}-1 &  1 & 1 &x_{m-1}-1 \\   * & g-1&  g-4 & g-1 \\  0 & * & 2 & 2
	 		\\ \hline\end{array}$$
			
\end{itemize}

{\bf II.3} $\quad \boldsymbol{c_{m}=2}$. In this case it is clear that $z_{m-1}=y_m=g-1$ (otherwise $c_m\ne 2$). Note also that if $y_{m-1}=0$, then $c_{m-1}\ne 2$ and then $c_m\ne 2.$ Thus, $y_{m-1}\ge 1$ and $c_{m-1}=2$. 
	$$\begin{array}{|c|cc|} \hline\delta_{m} & \delta_{m-1}&\delta_{m-2}\\  \hline 0 & 0 &* \\   y_{m-1}& g-1& y_{m-1}\\  * &g-1& g-1
	\\ \hline\end{array}\quad \longrightarrow \quad \begin{array}{|c|cc|} \hline\delta_{m} & \delta_{m-1}&\delta_{m-2}\\  \hline 1 & 1 &* \\   y_{m-1}-1& g-2& y_{m-1}-1\\  * &0& 0
	\\ \hline\end{array}$$
	Incidentally, this case only appears when $m=3$, so $l=6$. Indeed, for $l\ge 7$, we get $\delta_{m-1}=0$, which would make $n$ special, so Algorithm II does not apply to it.

\end{itemize}

\

 \

\

Let us illustrate this algorithm with an example.
We consider the positive integer representing the first $22$ decimal digits of  $e$:
$$
n=2718281828459045235360.
$$
First let us note that since  $\delta_{10}\ne 0$ and $\delta_{11}\ne 0$, then $n$ is a {\it normal integer}. In addition $n$ is of type A1. Therefore the initial configuration is:
$$\begin{array}{|ccccccccccc|ccccccccccc|}\hline 2 & 7 & 1 & 8 & 2 & 8 & 1& 8&2 & 8& 4 &5& 9 & 0 & 4 & 5&2&3&5& 3& 6 &0\\  \hline  2& . & . & . & . & . &. & . & .&. & .& . &.& . & . & .& .&.&.&.& .& 2\\ & 6 & . & . &. & . & . & .& .&. & .& . &.& . & . & . &.&.&.&.& .& 6\\ & & 2 &. &. & . & . & .& .&. & .& . &.& . & . & . &.&.&.&.& .& 2\\ \hline
\end{array}$$

Applying the algorithm II we get to the temporary configuration:
$$\begin{array}{|ccccccccccc|ccccccccccc|}\hline
2 & 7 & 1 & 8 & 2 & 8 & 1& 8&2 & 8& 4 &5& 9 & 0 & 4 & 5&2&3&5& 3& 6 &0\\  \hline 
 2& 0 & 1 & 1 & 1 & 0 & 1&0& 0&1 & 0& 0 &1& 0 & 0 & 1& 0&1&1&1& 0& 2\\ & 6 & 8 & 0 &0 & 3 & 1 & 7& 4&8& 2& 0 &2& 8 & 4 & 7 &1&3&0&0& 8& 6\\ 
 & &  2&7 &1 & 4 & 9 & 0& 7&9& 1& 5 &5& 1 & 9 & 7 &0&9&4&1& 7& 2\\ \hline\end{array}$$
Observe that the digit in column $12$ is not correct (we get a $3$ instead of a $4$ for the sum). This is because $c_{11}=0$, therefore we have to apply the adjustment step. 
Since $y_{11}=0,\ y_{10}\ne 0$ and $z_{10}\ne 0$, 
the adjustment step is that described in II.2.ii.a):
Applying algorithm II we get: 
$$\begin{array}{|c|ccccccccccc|ccccccccccc|}\hline n&2 & 7 & 1 & 8 & 2 & 8 & 1& 8&2 & 8& 4 &5& 9 & 0 & 4 & 5&2&3&5& 3& 6 &0\\  \hline  
p_1&2& 0 & 1 & 1 &  1 & 0 &  1 &0& 0& 1 & {\bf 1}&  {\bf 1} & 1& 0 & 0 & 1 & 0& 1&1 &1& 0& 2\\ p_2&& 6 & 8 & 0 &0 & 3 & 1 & 7& 4&8&  {\bf 1 }& {\bf 8}  & {\bf 1} & 8 & 4 & 7 &1&3 &0 &0 & 8 & 6 \\ 
p_3&& & 2 &7 &1 & 4 & 9 & 0& 7&9& 1& {\bf  6} & {\bf 6} & 1 & 9 & 7 &0&9&4&1& 7& 2\\ \hline\end{array}$$

		\subsection{Algorithm III} 
		
		 The cases $m\le 2$ correspond to the {\it small cases}. For $m\ge 3$, we proceed in the following steps:

		{\bf Step 1:} We choose $x_1,y_1,z_1$ according to the configurations described in Section \ref{initial}. Define $c_1=(1+y_1+z_1)/g,$ which is the carry of the column 1.
		
		\smallskip
		
		{\bf Step $2$:} Define the digits
		\begin{eqnarray*}  	
			x_2&=&\left \{\begin{matrix} D(\delta_{2m-2}-y_1) & \text{if } & \quad z_{1}\le \delta_{2m-3}-1;\\	D(\delta_{2m-2}-y_1-1)   & \text{ if } & z_{1}\ge \delta_{2m-3};	\end{matrix}\right .
			\\
			y_2&=& D( \delta_{2m-3}-z_{1}-1 ); \\
			z_2&=&D(\delta_{1}-x_1-y_{2}-c_{1});\\ c_2&=&(x_1+y_2+z_2+c_{1}-\delta_{1})/g\quad  (\text{the carry from column}\quad  2).	\end{eqnarray*}		
		
		{\bf Step $i,\ 3\le i\le m-1$:}  (these steps are vacuous for $m=3$). Define the digits
		\begin{eqnarray*}  
			x_i&=&\left\{\begin{matrix} 1 & \text{if } & \quad z_{i-1}\le \delta_{2m-i-1}-1;\\ 
				0  & \text{if} & z_{i-1}\ge \delta_{2m-i-1};
			\end{matrix} 
			\right.\\
			y_i&=&D( \delta_{2m-i-1}-z_{i-1}-1  ); \\
			z_i&=&D(\delta_{i-1}-x_{i-1}-y_{i}-c_{i-1});\\ c_i&=&(x_{i-1}+y_i+z_i+c_{i-1}-\delta_{i-1})/g\quad  (\text{the carry from column}\quad  i).
		\end{eqnarray*}
		
			{\bf Step $m$:}  Define the digits
			\begin{eqnarray*}  
				x_m&=&0.\\
				y_m&=&D( \delta_{m-1}-z_{m-1}-x_{m-1}-c_{m-1}  ).
			\end{eqnarray*}

		The temporary configuration is:
		$$\begin{array}{|ccccccc|ccccccccc|}\hline 1 & \delta_{2m-1} & \delta_{2m-2} & * & * & * & \delta_{m} & \delta_{m-1}&\delta_{m-2} & *& * &*& * & * & \delta_1 & \delta_0\\  \hline 1& x_1 & .& . & .& x_{m-1} & 0 & x_{m-1} &x_{m-2} & . &.& . &.  &.  & x_1 & 1\\  &  & y_1 & . & . & . &y_{m-1}& y_{m}&y_{m-1} & .& .& . & . & . & . & y_1\\  & &  & z_1& . & . & . & z_{m-1} & z_{m-1} & . & . & . & . & .& .& z_1\\ \hline 
		\end{array}$$

		We omit the proof of the following proposition because it is similar to the Proposition \ref{induction} of Algorithm I.
		\begin{prop}Let $g\ge 5$ and $m\ge 3$. We have that $\Delta_i=\delta_{i}$ for all $0\le i\le m-2$. Furthermore, for any $-1\le i\le m-2$, if $\Delta_{m+i}=\delta_{m+i}$, then $\Delta_{m+i+1}=\delta_{m+i+1}$.
		\end{prop}
		
		Again, the above proposition gives that if $\Delta_m=\delta_m$, then $\Delta_i=\delta_i$ for $i=0,\ldots,2m-1$.
		
		{\bf Adjustment step:} Notice that $\Delta_m\equiv \delta_m+c_m-1\pmod g$. According to this observation we distinguish the following cases:
		
		{\bf III.1} $\quad \boldsymbol{c_{m}=1}$. We do  nothing and the temporary configuration becomes the final one.

		{\bf III.2} $\quad \boldsymbol{c_{m}=0}$.				
		
		$$\begin{array}{|c|c|} \hline\delta_{m} & \delta_{m-1}\\  \hline 0 & * \\   *& *\\  * &*
		\\ \hline\end{array}\quad \longrightarrow \quad \begin{array}{|c|c|} \hline\delta_{m} & \delta_{m-1}\\  \hline 1 & * \\   *& *\\  * &*
		\\ \hline\end{array}$$

		{\bf III.3} $\quad \boldsymbol{c_{m}=2}$. Notice that $y_{m}\ne 0$ (otherwise $c_{m}\ne 2$). This is clear for $m\ge 4$ because $x_{m-1}$ takes the values $0$ or $1$. It also holds for $m=3$ because $x_2$ takes the values $0,1,2$ or $3$ for integers of type B when $g\ge 5$ and then $x_2\le g-2$.

			\begin{itemize}
				\item[III.3.i)] $\quad \boldsymbol{y_{m-1}\ne 0,\ \ z_{m-1}\ne g-1 }$. 
					$$\begin{array}{|c|cc|} \hline\delta_{m} & \delta_{m-1}&\delta_{m-2}\\  \hline 0 & * &* \\   y_{m-1}&y_m& y_{m-1}\\  * &z_{m-1}& z_{m-1}
					\\ \hline\end{array}\quad \longrightarrow \quad \begin{array}{|c|cc|} \hline\delta_{m} & \delta_{m-1}&\delta_{m-2}\\  \hline 0 & * &* \\   y_{m-1}-1&y_m-1& y_{m-1}-1\\  * &z_{m-1}+1& z_{m-1}+1
					\\ \hline\end{array}$$	
				\item[III.3.ii)] $\quad \boldsymbol{y_{m-1}\ne 0,\ \ z_{m-1}= g-1 }$. 
				$$\begin{array}{|c|cc|} \hline\delta_{m} & \delta_{m-1}&\delta_{m-2}\\  \hline 0 & * &* \\   y_{m-1}&y_m& y_{m-1}\\  * &g-1& g-1
				\\ \hline\end{array}\quad \longrightarrow \quad \begin{array}{|c|cc|} \hline\delta_{m} & \delta_{m-1}&\delta_{m-2}\\  \hline 1 & * &* \\   y_{m-1}-1&y_m
		& y_{m-1}-1\\  * &0& 0
				\\ \hline\end{array}$$
					\item[III.3.iii)] $\quad \boldsymbol{y_{m-1}= 0,\ \ z_{m-1}\ne g-1}$.  In this case $x_{m-1}\ne 0$.
					$$\begin{array}{|cc|cc|} \hline\delta_{m+1}&\delta_{m} & \delta_{m-1}&\delta_{m-2}\\  \hline x_{m-1}&0 &x_{m-1} &* \\  *& 0&y_m& 0\\ *&  * &z_{m-1}& z_{m-1}
					\\ \hline\end{array}\quad \longrightarrow \quad \begin{array}{|cc|cc|} \hline\delta_{m+1}&\delta_{m} & \delta_{m-1}&\delta_{m-2}\\  \hline x_{m-1}-1&0 &x_{m-1}-1 &* \\  *& g-1&y_m-1& g-1\\ *&  * &z_{m-1}+1& z_{m-1}+1
					\\ \hline\end{array}$$
					\item[III.3.iv)] $\quad \boldsymbol{y_{m-1}= 0,\ \ z_{m-1}= g-1}$.  In this case $x_{m-1}\ne 0$.
					$$\begin{array}{|cc|cc|} \hline\delta_{m+1}&\delta_{m} & \delta_{m-1}&\delta_{m-2}\\  \hline x_{m-1}&0 &x_{m-1} &* \\  *& 0&y_m& 0\\ *&  * &g-1& g-1
					\\ \hline\end{array}\quad \longrightarrow \quad \begin{array}{|cc|cc|} \hline\delta_{m+1}&\delta_{m} & \delta_{m-1}&\delta_{m-2}\\  \hline x_{m-1}-1&1 &x_{m-1}-1 &* \\  *& g-1&y_m& g-1\\ *&  * &0& 0
					\\ \hline\end{array}$$
			\end{itemize}

		{\bf Example:}	Let us illustrate this algorithm with an example. 
			We consider the positive integer representing the first $21$ decimal digits of  $\zeta(3)$:
			$$
			n=	120205690315959428539.
			$$
				
			First let us note that $n$  is a {\it normal integer} because the number of digits is odd. In addition $n$ is of type B.5. Therefore the initial configuration is:	$$\begin{array}{|ccccccccccc|cccccccccc|}\hline 1 & 2 & 0 &2 & 0& 5 & 6& 9&0 & 3& 1 &5& 9 & 5 & 9 & 4&2&8&5& 3&  9\\  \hline  \bf 1& \bf 1 & . & . & . & . &. & . & .&. & .& . &.& . & .& .&.&.&.& \bf1& \bf1\\ &  & \bf9 & . &. & . &  .& .&. & .& . &.& . & . & . &.&.&.&.& .&\bf 9\\ & &  &\bf9 & . & . & .& .&. & .& . &.& . & . & . &.&.&.&.& .& \bf 9\\ \hline
			\end{array}$$
			Applying the algorithm III we get to the temporary configuration.	
			Since $c_{10}=1$ we do not need any adjustment step and the temporary configuration is also the final configuration.
			$$\begin{array}{|c|ccccccccccc|cccccccccc|}\hline n&1 & 2 & 0 &2 & 0& 5 & 6& 9&0 & 3& 1 &5& 9 & 5 & 9 & 4&2&8&5& 3& 9\\  \hline p_1& 1& 1 & 0& 0 &1 & 0 &1 &0 & 0&1 & 0& 1 &0& 0 &1 & 0& 1&0&0&1& 1\\ p_2& & & 9 & 2 &0 & 0 & 7 & 4& 0&5 & 0& 5 &0& 5  &0 &4&7&0&0& 2& 9\\ p_3 & & & &9 &9 & 4 & 8 & 4& 9&7& 0 &9& 9 & 0 & 7 &9&4&8&4& 9& 9\\ \hline
			\end{array}$$

			\subsection{Algorithm IV}  
			
		 The cases $m\le 3$ correspond to the {\it small cases}. For $m\ge 4$, we proceed in the following steps:

				{\bf Step 1:} We choose $x_1,y_1,z_1$ according to the configurations described in Section \ref{initial}. Define $c_1=(1+y_1+z_1)/g,$ which is the carry of the column 1.
				
				\smallskip
				
				{\bf Step $2$:} Define the digits
				\begin{eqnarray*}  	
					x_2&=&\left \{\begin{matrix} D(\delta_{2m-3}-y_1) & \text{if } & \quad z_{1}\le \delta_{2m-4}-1;\\	D(\delta_{2m-3}-y_1-1)   & \text{ if } & z_{1}\ge \delta_{2m-4};	\end{matrix}\right .
					\\
					y_2&=& D( \delta_{2m-4}-z_{1}-1 ); \\
					z_2&=&D(\delta_{1}-x_1-y_{2}-c_{1});\\ c_2&=&(x_1+y_2+z_2+c_{1}-\delta_{1})/g\quad  (\text{the carry from column}\quad  2).
				\end{eqnarray*}		
				
				{\bf Step $i,\ 3\le i\le m-2$:}  Define the digits
				\begin{eqnarray*}  
					x_i&=&\left\{\begin{matrix} 1 & \text{if } & \quad z_{i-1}\le \delta_{2m-i-2}-1;\\ 
						0  & \text{if} & z_{i-1}\ge \delta_{2m-i-2};
					\end{matrix} 
					\right.\\
					y_i&=&D( \delta_{2m-i-2}-z_{i-1}-1  ); \\
					z_i&=&D(\delta_{i-1}-x_{i-1}-y_{i}-c_{i-1});\\ c_i&=&(x_{i-1}+y_i+z_i+c_{i-1}-\delta_{i-1})/g\quad  (\text{the carry from column}\quad  i).
				\end{eqnarray*}
				
					{\bf Step $i=m-1$:}  Define the digits
					\begin{eqnarray*}  
						x_{m-1}&=&\left\{\begin{matrix} 1 & \text{if } & \quad z_{m-2}\le \delta_{m-1}-1;\\ 
							0  & \text{if} & z_{m-2}\ge \delta_{m-1};
						\end{matrix} 
						\right.\\
						y_{m-1}&=&D( \delta_{m-1}-z_{m-2}-1  )\\
						z_{m-1}&=&D(\delta_{m-2}-x_{m-2}-y_{m-1}-c_{m-2}).
					\end{eqnarray*}

			The temporary configuration is:
		$$\begin{array}{|cccccccc|cccccccc|}\hline1 & \delta_{2m-2} & \delta_{2m-3} & * & * & * & \delta_{m} & \delta_{m-1}&\delta_{m-2} & *& * &*& * & * & \delta_1 & \delta_0\\  \hline 1 & x_1 & .& . & .& x_{m-2} & x_{m-1} & x_{m-1} &x_{m-2} & . &.& . &.  &.  & x_1 & 1\\  &  & y_1 & . & . & . &y_{m-2}& y_{m-1}&y_{m-1} & y_{m-2}& .& . & . & . & . & y_1\\  & &  & z_1& . & . & . & z_{m-2} & z_{m-1} & z_{m-2} & . & . & . & .& .& z_1\\ \hline	\end{array}$$

						\begin{prop}Let $g\ge 5
							$ and $m\ge 4$. We have that $\Delta_i=\delta_{i}$ for all $0\le i\le m-2$. Furthermore, for any $-1\le i\le m-3$, if $\Delta_{m+i}=\delta_{m+i}$, then $\Delta_{m+i+1}=\delta_{m+i+1}$.
							\end{prop}\begin{proof} The first statement of the proposition is clear. For the second one, we consider first the case $i=-1$.							
							Assuming that $\Delta_{m-1}=\delta_{m-1}$ we have to prove that $\Delta_m=\delta_m$. Indeed $$\Delta_m\equiv x_{m-1}+y_{m-2}+z_{m-3}+c_m\equiv \delta_m+x_{m-1}+c_m-1\pmod g.$$
							If $x_{m-1}=1$ then $z_{m-2}\le \delta_{m-1}-1$ and $y_{m-1}=\delta_{m-1}-z_{m-2}-1$. On the other hand, since $\Delta_{m-1}\equiv \delta_{m-1}+x_{m-1}+c_{m-1}-1\pmod g$ and $\Delta_{m-1}=\delta_{m-1}$,  we have that
							$x_{m-1}+c_{m-1}=1$. Thus, $c_{m-1}=0$. Finally 
							$$c_m=(x_{m-1}+y_{m-1}+z_{m-2}+c_{m-1}-\delta_{m-1})/g=0.$$
							If $x_{m-1}=0$, then $z_{m-2}\ge \delta_{m-1}$ and $y_{m-1}=g+\delta_{m-1}-z_{m-2}-1$. On the other hand, since $\Delta_{m-1}\equiv \delta_{m-1}+x_{m-1}+c_{m-1}-1\pmod g$ and $\Delta_{m-1}=\delta_{m-1}$, we have that
							$x_{m-1}+c_{m-1}=1$. Thus $c_{m-1}=1$. Finally 
							$$c_m=(x_{m-1}+y_{m-1}+z_{m-2}+c_{m-1}-\delta_{m-1})/g=1.$$
							In any case we have that $x_{m-1}+c_m=1$ and then we conclude that $\Delta_m=\delta_m$.
							
							We omit the proof of the proposition for the other cases because they are similar to the case $i=-1$.\end{proof}
							
							The above proposition gives that if $\Delta_{m-1}=\delta_{m-1}$ then $\Delta_i=\delta_i$ for all $i=0, \ldots,2m-2.$
							
							The adjustment step of this algorithm is more complicated than the previous ones.
							
								{\bf Adjustment step:} Assume that $m\ge 4$.	Notice that in this algorithm we have that
								$$\Delta_{m-1}\equiv \delta_{m-1}+x_{m-1}+c_{m-1}-1\pmod g.$$
								
								{\bf IV.1} $\quad \boldsymbol{x_{m-1}+c_{m-1}=1}$. We do  nothing and the temporary configuration becomes the final one.
								
								{\bf IV.2} $\quad \boldsymbol{x_{m-1}+c_{m-1}=0,\ y_{m-1}\ne g-1}$. Then $x_{m-1}=c_{m-1}=0$. If $y_{m-1}=0$, then $z_{m-2}\equiv \delta_{m-2}-1\pmod g$, thus $z_{m-1}\le \delta_{m-2}-1$, so $x_{m-1}=1$ unless $\delta_{m-1}=0$, which is not allowed. Thus, $y_{m-1}\ne 0$.
								
								\begin{itemize}
									\item[IV.2.i)] $\quad \boldsymbol{ z_{m-1}\ne 0.}$
									$$\begin{array}{|c|c|} \hline\delta_{m-1} & \delta_{m-2}\\  \hline  * & *\\  y_{m-1}&y_{m-1}\\ * & z_{m-1}\\ \hline
									\end{array}\quad \longrightarrow \quad \begin{array}{|c|c|} \hline\delta_{m-1} & \delta_{m-2}\\  \hline * & * \\  y_{m-1}+1&y_{m-1}+1\\ * & z_{m-1}-1\\ \hline
									\end{array}$$
									
									\item[IV.2.ii)] $\quad \boldsymbol{  z_{m-1}=0,\, y_{m-2}\ne 0.}$
									
									\smallskip
									
									\begin{itemize}
										\item[IV.2.ii.a)] $\quad \boldsymbol {y_{m-1}\ne 1,\, z_{m-2}\ne g-1}$.
										$$\begin{array}{|cc|cc|} \hline\delta_m&\delta_{m-1} & \delta_{m-2}&*\\  \hline 0&0& * & *\\  y_{m-2}&y_{m-1}&y_{m-1}&y_{m-2}\\ *&z_{m-2} & 0&z_{m-2}\\ \hline
										\end{array}\quad \longrightarrow \quad \begin{array}{|cc|cc|} \hline\delta_m&\delta_{m-1} & \delta_{m-2}&*\\  \hline 1&1& * & *\\  y_{m-2}-1&y_{m-1}-1&y_{m-1}-1&y_{m-2}-1\\ *&z_{m-2}+1 & 1&z_{m-2}+1\\ \hline
										\end{array}$$
										
										\item[IV.2.ii.b)] $\quad \boldsymbol {y_{m-1}\ne 1,\, z_{m-2}=g-1}$.
										Note that $y_{m-1}\ne 0$ since otherwise, $y_{m-1}=0$ would imply $\delta_{m-1}=0$, which is false.
										$$\begin{array}{|cc|cc|} \hline\delta_m&\delta_{m-1} & \delta_{m-2}&*\\  \hline 0&0& * & *\\  y_{m-2}&y_{m-1}&y_{m-1}&y_{m-2}\\ *&g-1 & 0&g-1\\ \hline
										\end{array}\quad \longrightarrow \quad \begin{array}{|cc|cc|}\hline \delta_m&\delta_{m-1} & \delta_{m-2}&*\\  \hline 2&2& * & *\\  y_{m-2}-1&y_{m-1}-2&y_{m-1}-2&y_{m-2}-1\\ *&0 & 3&0\\ \hline
										\end{array}$$

										\item[IV.2.ii.c)] $\quad \boldsymbol {y_{m-1}=1}$. In this case, since $y_{m-1}+z_{m-2}+1\equiv \delta_{m-1}\pmod g$, we get that $z_{m-2}=g-1$. Indeed, for if not then either
										$z_{m-2}=g-2$, giving $\delta_{m-1}=0$, which is not allowed, or $z_{m-2}\le g-3$, giving $z_{m-2}=\delta_{m-1}-2$, which contradicts the fact that $x_{m-1}=0$. We make the following adjustment:  
										$$\begin{array}{|cc|cc|} \hline\delta_m&\delta_{m-1} & \delta_{m-2}&*\\  \hline 0&0&* & *\\  y_{m-2}&1&1&y_{m-2}\\ *&g-1 & 0&g-1\\ \hline
										\end{array}\quad \longrightarrow \quad \begin{array}{|cc|cc|} \hline\delta_m&\delta_{m-1} & \delta_{m-2}&*\\  \hline 1&1& * & *\\  y_{m-2}-1&g-1&g-1&y_{m-2}-1\\ *&0 & 3&0\\ \hline
										\end{array}$$				
									\end{itemize}
									\item[IV.2.iii)] $\quad \boldsymbol{z_{m-1}=0,\ y_{m-2}= 0.}$ Notice that $y_{m-2}\equiv \delta_m-z_{m-3}-1 \pmod g$. Since $y_{m-2}=0$ and $\delta_m\ne 0$, we have that $z_{m-3}\le \delta_m-1$ and then $x_{m-2}\ne 0$ (even when $m=4$).
									
									\smallskip

									\begin{itemize}
										\item[IV.2.iii.a)] $\quad \boldsymbol{z_{m-2}\ne g-1}$. It follows that $y_{m-1}\ne 0$. Otherwise we would have $\delta_{m-1}= 0$, which is not allowed.
										$$\begin{array}{|ccc|cc|} \hline*&\delta_m&\delta_{m-1} & \delta_{m-2}&*\\  \hline x_{m-2}&0&0&x_{m-2}& *\\  *&0&y_{m-1}&y_{m-1}&0\\ *&*&z_{m-2}& 0&z_{m-2}\\ \hline
										\end{array}\quad \longrightarrow \quad \begin{array}{|ccc|cc|} \hline*&\delta_m&\delta_{m-1} & \delta_{m-2}&*\\  \hline x_{m-2}-1&1&1& x_{m-2}-1
										& *\\ *& g-1&y_{m-1}-1&y_{m-1}-1&g-1\\ *&*&z_{m-2}+1 & 1&z_{m-2}+1\\ \hline
										\end{array}$$
										\item[IV.2.iii.b)] $\quad \boldsymbol{z_{m-2}= g-1,\ y_{m-1}\ne 1} $.
										$$\begin{array}{|ccc|cc|} \hline*&\delta_m&\delta_{m-1} & \delta_{m-2}&*\\  \hline x_{m-2}&0&0&x_{m-2}& *\\  *&0&y_{m-1}&y_{m-1}&0\\ *&*&g-1& 0&g-1\\ \hline
										\end{array}\quad \longrightarrow \quad \begin{array}{|ccc|cc|} \hline*&\delta_m&\delta_{m-1} & \delta_{m-2}&*\\  \hline x_{m-2}-1&2&2& x_{m-2}-1&*\\ *& g-1&y_{m-1}-2&y_{m-1}-2&g-1\\ *&*&0& 3&0\\ \hline
										\end{array}$$
										\item[IV.2.iii.c)] $\quad \boldsymbol{z_{m-2}= g-1,\ y_{m-1}= 1} $.
										$$\begin{array}{|ccc|cc|} \hline*&\delta_m&\delta_{m-1} & \delta_{m-2}&*\\  \hline x_{m-2}&0&0&x_{m-2}& *\\  *&0&1&1&0\\ *&*&g-1& 0&g-1\\ \hline
										\end{array}\quad \longrightarrow \quad \begin{array}{|ccc|cc|} \hline*&\delta_m&\delta_{m-1} & \delta_{m-2}&*\\  \hline x_{m-2}-1&1&1& x_{m-2}-1&*\\ *& g-1&g-1&g-1&g-1\\ *&*&0& 3&0\\ \hline
										\end{array}$$

									\end{itemize}

								\end{itemize}
								
								\
								
								{\bf IV.3} $\quad \boldsymbol{x_{m-1}+c_{m-1}=0,}\ \quad \boldsymbol{y_{m-1}=g-1}$.
								  Since $c_{m-1}=0$, it follows that $x_{m-2}=z_{m-1}=0$.  Notice that if $y_{m-2}=0$, then $\delta_m=0$ (otherwise $z_{m-3}=\delta_m-1$ and then $x_{m-2}\ne 0$), which is not allowed. Thus, $y_{m-2}\ne 0$. Further, if $z_{m-2}=g-1$, then $c_{m-2}=(x_{m-3}+y_{m-2}+z_{m-2})/g\ge (x_{m-3}+1+g-1)/g\ge 1$, so $c_{m-1}=(x_{m-2}+g-1+c_{m-2})/g\ge 1$, a contradiction. Thus, $\boldsymbol{z_{m-2}\ne g-1}$ and we make the following adjustment: 
								        $$\begin{array}{|cc|cc|} \hline\delta_m&\delta_{m-1} & \delta_{m-2}&*\\  \hline   0&0&*& *\\  y_{m-2}&g-1&g-1&y_{m-2}\\ *&z_{m-2}& 0&z_{m-2}\\ \hline
									\end{array}\quad \longrightarrow \quad \begin{array}{|cc|cc|} \hline\delta_m&\delta_{m-1} & \delta_{m-2}&*\\  \hline 1&1& * & *\\  y_{m-2}-1&g-2&g-2&y_{m-2}-1\\ *&z_{m-2}+1 & 1&z_{m-2}+1\\ \hline
									\end{array}$$

								\bigskip
								{\bf IV.4} $\quad \boldsymbol{x_{m-1}+c_{m-1}=2,\ x_{m-1}=0,\ c_{m-1}=2}$. 		If $y_{m-1}=0$, then $z_{m-2}=g-1$ and then $\delta_{m-1}\ne 0$. So, $y_{m-1}\ne 0$.
								\begin{itemize}
									\item[IV.4.i)] $\quad \boldsymbol{z_{m-1}\ne g-1}$. 
									$$\begin{array}{|c|c|} \hline\delta_{m-1} & \delta_{m-2}\\  \hline  * & * \\  y_{m-1}&y_{m-1}\\ z_{m-2} & z_{m-1}\\ \hline
									\end{array}\quad \longrightarrow \quad \begin{array}{|c|c|} \hline\delta_{m-1} & \delta_{m-2}\\  \hline  * & * \\  y_{m-1}-1&y_{m-1}-1\\ z_{m-2} & z_{m-1}+1\\ \hline
									\end{array}$$
									
									\item[IV.4.ii)] $\quad \boldsymbol{z_{m-1}= g-1,\ z_{m-2}\ne g-1}$. Notice that $y_{m-1}\ne 1$. Otherwise $c_{m-1}\ne 2$ (even when $m=4$)
									
									\smallskip

									\begin{itemize}
										\item[IV.4.ii.a)] $\quad \boldsymbol{y_{m-2}\ne 0}$.	
										$$\begin{array}{|cc|cc|} \hline\delta_m&\delta_{m-1} & \delta_{m-2}&*\\  \hline  0&0&*& *\\  y_{m-2}&y_{m-1}&y_{m-1}&y_{m-2}\\ *&z_{m-2}& g-1&z_{m-2}\\ \hline
										\end{array}\quad \longrightarrow \quad \begin{array}{|cc|cc|} \hline\delta_m&\delta_{m-1} & \delta_{m-2}&*\\  \hline  1&1&*&*\\  y_{m-2}-1&y_{m-1}-2&y_{m-1}-2&y_{m-2}-1\\ *&z_{m-2}+1& 1&z_{m-2}+1\\ \hline
										\end{array}$$
										\item[IV.4.ii.b)] $\quad \boldsymbol{y_{m-2}= 0}$. As in case IV.2.iii), we have that $x_{m-2}\ne 0$.
										$$\begin{array}{|ccc|cc|} \hline*&\delta_m&\delta_{m-1} & \delta_{m-2}&*\\  \hline  x_{m-2}&0&0&x_{m-2}&*\\  *&0&y_{m-1}&y_{m-1}&0\\ *&*&z_{m-2}& g-1&z_{m-2}\\ \hline
										\end{array}\quad \longrightarrow \quad \begin{array}{|ccc|cc|} \hline*&\delta_m&\delta_{m-1} & \delta_{m-2}&*\\  \hline  x_{m-2}-1&1&1&x_{m-2}-1& *\\  *&g-1&y_{m-1}-2&y_{m-1}-2&g-1\\ *&*&z_{m-2}+1& 1&z_{m-2}+1\\ \hline
										\end{array}$$
										\item[IV.4.iii)]  $\quad \boldsymbol{z_{m-1}= g-1,\ z_{m-2}= g-1}$. 
										In this case, we make the following adjustments:
										
										\medskip 
										
										\begin{itemize}
										\item[IV.4.iii.a)] $y_{m-1}\not\in \{g-1,g-2\}$. In this case, $x_{m-2}\ge 1$, otherwise  the sum in the column $m-1$ is at most $z_{m-1}+y_{m-1}+x_{m-2}+c_{m-2}\le g-1+g-3+0+2=2g-2<2g$, so we cannot have $c_{m-1}=2$. If $y_{m-2}\ne g-1$, then
										$$\begin{array}{|ccc|cc|} \hline \delta_{m+1} & \delta_m&\delta_{m-1} & \delta_{m-2}&*\\  \hline  x_{m-2} & 0&0& x_{m-2} & *\\   * & y_{m-2}&y_{m-1}&y_{m-1}&y_{m-2}\\  * & *& g-1& g-1& g-1\\ \hline
										\end{array}\quad \longrightarrow \quad \begin{array}{|ccc|cc|} \hline*&\delta_m&\delta_{m-1} & \delta_{m-2}&*\\  \hline  x_{m-2}-1&g-2&g-2&x_{m-2}-1& *\\  *&y_{m-2}+1&y_{m-1}+2&y_{m-1}+2&y_{m-2}+1\\ *&*&g-2& g-2&g-2\\ \hline
										\end{array}$$
										while if $y_{m-2}=g-1$, then  
$$\begin{array}{|ccc|cc|} \hline \delta_{m+1} & \delta_m&\delta_{m-1} & \delta_{m-2}&*\\  \hline  x_{m-2} & 0&0& x_{m-2} & *\\   * & g-1&y_{m-1}&y_{m-1}&g-1\\  * & *& g-1& g-1& g-1\\ \hline
										\end{array}\quad \longrightarrow \quad \begin{array}{|ccc|cc|} \hline*&\delta_m&\delta_{m-1} & \delta_{m-2}&*\\  \hline  x_{m-2}&g-2&g-2&x_{m-2}& *\\  *& 0 &y_{m-1}+2&y_{m-1}+2& 0 \\ *&*&g-2& g-2&g-2\\ \hline
										\end{array}$$
										\item[IV.4.iii.b)] If $y_{m-1}\in \{g-1,g-2\}$, then if $y_{m-2}\ge 1$:
										$$\begin{array}{|ccc|cc|} \hline \delta_{m+1} & \delta_m&\delta_{m-1} & \delta_{m-2}&*\\  \hline  x_{m-2} & 0&0& x_{m-2} & *\\   * & y_{m-2}&y_{m-1}&y_{m-1}&y_{m-2}\\  * & *& g-1& g-1& g-1\\ \hline
										\end{array}\quad \longrightarrow \quad \begin{array}{|ccc|cc|} \hline*&\delta_m&\delta_{m-1} & \delta_{m-2}&*\\  \hline  x_{m-2}&2&2&x_{m-2}& *\\  *&y_{m-2}-1&y_{m-1}-3&y_{m-1}-3&y_{m-2}-1\\ *&*&0& 3& 0\\ \hline
										\end{array}$$
										If $y_{m-2}=0$ but $x_{m-2}\ge 1$, then
										$$\begin{array}{|ccc|cc|} \hline \delta_{m+1} & \delta_m&\delta_{m-1} & \delta_{m-2}&*\\  \hline  x_{m-2} & 0&0& x_{m-2} & *\\   * & 0&y_{m-1}&y_{m-1}&0\\  * & *& g-1& g-1& g-1\\ \hline
										\end{array}\quad \longrightarrow \quad \begin{array}{|ccc|cc|} \hline*&\delta_m&\delta_{m-1} & \delta_{m-2}&*\\  \hline  x_{m-2}-1&2&2&x_{m-2}-1& *\\  *&g-1&y_{m-1}-3&y_{m-1}-3&g-1\\ *&*&0& 3&0\\ \hline
										\end{array}$$ 										
										\end{itemize}
										This exhausts all possibilities. Indeed, if $y_{m-2}=x_{m-2}=0$, then since $c_{m-1}=2$, the only possibility is that $c_{m-2}=2$, which implies that $x_{m-3}=g-1$, which is false since $x_i\le 2$ for all $i\ge 1$. 
									\end{itemize}
									
								\end{itemize}
								
								\
								
								{\bf IV.5} $\quad \boldsymbol{x_{m-1}+c_{m-1}=2,\ x_{m-1}=1,\ c_{m-1}=1}$. In particular, it follows that $z_{m-2}\ne g-1$ (otherwise we would have $x_{m-1}=0$). Also, $y_{m-1}\ne g-1$. 
								Indeed, since $y_{m-1}+z_{m-2}+1\equiv \delta_{m-1}\pmod g$, it follows that if $y_{m-1}=g-1$, then $z_{m-2}=\delta_{m-1}$, so $x_{m-1}=0$, a contradiction. 
								\begin{itemize}
									\item[IV.5.i)] $\quad \boldsymbol{z_{m-1}\ne g-1,\ y_{m-1}\ne 0}$.
									$$\begin{array}{|c|c|} \hline\delta_{m-1} & \delta_{m-2}\\  \hline  * & * \\  y_{m-1}&y_{m-1}\\ * & z_{m-1}\\ \hline
									\end{array}\quad \longrightarrow \quad \begin{array}{|c|c|} \hline\delta_{m-1} & \delta_{m-2}\\  \hline  * & * \\  y_{m-1}-1&y_{m-1}-1\\ * & z_{m-1}+1\\ \hline
									\end{array}$$
									\item[IV.5.ii)] $\quad \boldsymbol{z_{m-1}\ne g-1,\ y_{m-1}= 0}$.
									$$\begin{array}{|cc|c|} \hline*&\delta_{m-1} & \delta_{m-2}\\  \hline  1&1 &* \\  *&0&0\\ *&* & z_{m-1}\\ \hline
									\end{array}\quad \longrightarrow \quad \begin{array}{|cc|c|} \hline*&\delta_{m-1} & \delta_{m-2}\\  \hline  0&0& * \\  *&g-1&g-1\\ *&*& z_{m-1}+1\\ \hline
									\end{array}$$	
									
									\
									
									\item[IV.5.iii)] $\quad \boldsymbol{z_{m-1}= g-1,\ z_{m-2}\ne 0}$.			
									
									\smallskip			 
									
									\begin{itemize}
										\item[IV.5.iii.a)] $\quad \boldsymbol{y_{m-2}\ne g-1}$.			
										$$\begin{array}{|cc|cc|} \hline\delta_m&\delta_{m-1} & \delta_{m-2}&*\\  \hline  1&1&*& *\\  y_{m-2}&y_{m-1}&y_{m-1}&y_{m-2}\\ *&z_{m-2}& g-1&z_{m-2}\\ \hline
										\end{array}\quad \longrightarrow \quad \begin{array}{|cc|cc|} \hline\delta_m&\delta_{m-1} & \delta_{m-2}&*\\  \hline  0&0&*& *\\  y_{m-2}+1&y_{m-1}+1&y_{m-1}+1&y_{m-2}+1\\ *&z_{m-2}-1& g-2&z_{m-2}-1\\ \hline
										\end{array}$$
										\item[IV.5.iii.b)] $\quad \boldsymbol{y_{m-2}= g-1,\ y_{m-1}\ne 0,1}$.			
										$$\begin{array}{|cc|cc|} \hline\delta_m&\delta_{m-1} & \delta_{m-2}&*\\  \hline 1&1&*& *\\  g-1&y_{m-1}&y_{m-1}&g-1\\ *&z_{m-2}& g-1&z_{m-2}\\ \hline
										\end{array}\quad \longrightarrow \quad \begin{array}{|cc|cc|} \hline\delta_m&\delta_{m-1} & \delta_{m-2}&*\\  \hline 2&2&*& *\\  g-2&y_{m-1}-2&y_{m-1}-2&g-2\\ *&z_{m-2}+1& 1&z_{m-2}+1\\ \hline
										\end{array}$$
										\item[IV.5.iii.c)] $\quad \boldsymbol{y_{m-2}= g-1,\ y_{m-1}=0}$.			
										$$\begin{array}{|cc|cc|} \hline\delta_m&\delta_{m-1} & \delta_{m-2}&*\\  \hline 1&1&*& *\\  g-1&0&0&g-1\\ *&z_{m-2}& g-1&z_{m-2}\\ \hline
										\end{array}\quad \longrightarrow \quad \begin{array}{|cc|cc|} \hline\delta_m&\delta_{m-1} & \delta_{m-2}&*\\  \hline 1&1&*& *\\  g-2&g-2&g-2&g-2\\ * &z_{m-2}+1& 1&z_{m-2}+1
										\\ \hline\end{array}$$	
										\item[IV.5.iii.d)] $\quad \boldsymbol{y_{m-2}= g-1,\ y_{m-1}=1}$.			
										$$\begin{array}{|cc|cc|} \hline\delta_m&\delta_{m-1} & \delta_{m-2}&*\\  \hline  1&1&*& *\\  g-1&1&1&g-1\\ *&z_{m-2}& g-1&z_{m-2}\\ \hline
										\end{array}\quad \longrightarrow \quad \begin{array}{|cc|cc|} \hline\delta_m&\delta_{m-1} & \delta_{m-2}&*\\  \hline  1&1&*& *\\  g-2&g-1&g-1&g-2\\ *&z_{m-2}+1& 1&z_{m-2}+1\\ \hline
										\end{array}$$								\end{itemize}
									
									\
									
									\item[IV.5.iv)] $\quad \boldsymbol{z_{m-1}= g-1,\ z_{m-2}=0,\ y_{m-2}\ne 0}$.			
									
									\smallskip	
									
									\begin{itemize}
										
										\item[IV.5.iv.a)] $\quad \boldsymbol{ y_{m-1}\ne 0,1}$. 
										$$\begin{array}{|cc|cc|} \hline\delta_m&\delta_{m-1} & \delta_{m-2}&*\\  \hline 1&1&*& *\\ y_{m-2}&y_{m-1}&y_{m-1}&y_{m-2}\\ *&0& g-1&0\\ \hline
										\end{array}\quad \longrightarrow \quad \begin{array}{|cc|cc|} \hline\delta_m&\delta_{m-1} & \delta_{m-2}&*\\  \hline 2&2&*& *\\  y_{m-2}-1&y_{m-1}-2&y_{m-1}-2&y_{m-2}-1\\ *&1& 1&1\\ \hline
										\end{array}$$
										\item[IV.5.iv.b)] $\quad \boldsymbol{ y_{m-1}=0}$. 
										$$\begin{array}{|cc|cc|} \hline\delta_m&\delta_{m-1} & \delta_{m-2}&*\\  \hline  1&1&*& *\\ y_{m-2}&0&0&y_{m-2}\\ *&0& g-1&0\\ \hline
										\end{array}\quad \longrightarrow \quad \begin{array}{|cc|cc|} \hline\delta_m&\delta_{m-1} & \delta_{m-2}&*\\  \hline  1&1&*& *\\  y_{m-2}-1&g-2&g-2&y_{m-2}-1\\*&1& 1&1\\ \hline
										\end{array}$$	
										\item[IV.5.iv.c)] $\quad \boldsymbol{ y_{m-1}=1}$. 
										$$\begin{array}{|cc|cc|} \hline\delta_m&\delta_{m-1} & \delta_{m-2}&*\\  \hline  1&1& * &  * \\ y_{m-2}&1&1&y_{m-2}\\  * &0& g-1&0\\ \hline
										\end{array}\quad \longrightarrow \quad \begin{array}{|cc|cc|} \hline\delta_m&\delta_{m-1} & \delta_{m-2}&*\\  \hline 1&1&*& *\\  y_{m-2}-1&g-1&g-1&y_{m-2}-1\\ *&1& 1&1\\ \hline
										\end{array}$$				
									\end{itemize}
									
									\smallskip
									
									\item[IV.5.v)] $\quad \boldsymbol{z_{m-1}= g-1,\ z_{m-2}=0,\ y_{m-2}= 0}$.	If $x_{m-2}=0$, then $\delta_m=0$, which is not allowed. Thus, $x_{m-2}\ne 0$ (even when $m=4$).		
									
									\smallskip											\begin{itemize}
										\item[IV.5.v.a)] $\quad \boldsymbol{y_{m-1}\ne 0,1}$. As in case IV.2.iii), we have that $x_{m-2}\ne 0$.
										$$\begin{array}{|ccc|cc|} \hline*&\delta_m&\delta_{m-1} & \delta_{m-2}&*\\  \hline x_{m-2}&1&1&x_{m-2}& *\\  *&0&y_{m-1}&y_{m-1}&0\\ * & *&0& g-1&0\\ \hline
										\end{array}\quad \longrightarrow \quad \begin{array}{|ccc|cc|} \hline*&\delta_m&\delta_{m-1} & \delta_{m-2}&*\\  \hline x_{m-2}-1&2&2& x_{m-2}-1 & *\\  *& g-1&y_{m-1}-2&y_{m-1}-2&g-1\\ * & * &1 & 1&1\\ \hline
										\end{array}$$	
										
										\item[IV.5.v.b)] $\quad \boldsymbol{y_{m-1}=0}$. Note that $x_{m-2}=1$. Indeed, if $x_{m-2}=0$, then in order to have $c_{m-1}=1$, we would need 
										that $c_{m-2}=1$, so $x_{m-3}\ge g-2$, which is false. 
										$$\begin{array}{|ccc|cc|} \hline \delta_{m+1} & \delta_m&\delta_{m-1} & \delta_{m-2}&*\\  \hline  x_{m-2} & 1&1& x_{m-2} & *\\ * & 0&0&0&0\\   * & *&0& g-1&0\\ \hline
										\end{array}\quad \longrightarrow \quad \begin{array}{|ccc|cc|} \hline \delta_{m+1} & \delta_m&\delta_{m-1} & \delta_{m-2}&*\\  \hline x_{m-2}-1 & 1&1& x_{m-2}-1 & *\\  * & g-1&g-2&g-2&g-1\\  
										 * & *&1 & 1 &1\\ \hline
										\end{array}$$								\item[IV.5.v.c)] $\quad \boldsymbol{y_{m-1}=1}$. Then $x_{m-2}=1$. Indeed, this follows as before, namely 
										since $y_{m-2}\equiv \delta_m-z_{m-3}-1\pmod g$, and $\delta_m\ne 0$, $y_{m-2}=0$, we get that $z_{m-3}\le \delta_m-1$, so $x_{m-2}\ne 0$ (even when $m=4$). Then  
										 
										$$\begin{array}{|ccc|cc|} \hline \delta_{m+1} & \delta_m&\delta_{m-1} & \delta_{m-2}&*\\  \hline x_{m-2} & 1&1& x_{m-2} & *\\  * & 0&1&1&0\\  * &*&0& g-1&0\\ \hline
										\end{array}\quad \longrightarrow \quad \begin{array}{|ccc|cc|} \hline \delta_{m+1} & \delta_m&\delta_{m-1} & \delta_{m-2}&*\\  \hline   x_{m-2}-1 & 1&1& x_{m-2}-1 & *\\ * & g-1&g-1	&g-1&g-1\\ * & *&1 & 1&1\\ \hline
										\end{array}$$										\end{itemize}											
								\end{itemize}
								
								\
								
								{\bf IV.6} $\quad \boldsymbol{x_{m-1}+c_{m-1}=3}$. Then $x_{m-1}=1$ and $c_{m-1}=2.$ We always have that $x_{m-2}\le 3$ (even when $m=4$). It follows that $y_{m-1}\ge 1$ and $z_{m-1}= g-1$ (otherwise $z_{m-1}+y_{m-1}+x_{m-2}+c_{m-2}\le g-1+4+2\le 2g-1$ and then $c_{m-1}\ne 2$).									
								$$\begin{array}{|c|c|} \hline\delta_{m-1} & \delta_{m-2}\\  \hline *&* \\  y_{m-1}&y_{m-1}\\ * & g-1\\ \hline
								\end{array}\quad \longrightarrow \quad \begin{array}{|c|c|} \hline\delta_{m-1} & \delta_{m-2}\\  \hline  * & *
								\\  y_{m-1}-1&y_{m-1}-1\\ * & 0\\ \hline
								\end{array}$$

				\subsection{Algorithm V}

		We recall that in this case the associated palindrome $p_1$ of $n$ has $2m$ digits and that $\delta_{m-1}=0$ or $\delta_m=0$. First we consider  the integer 
		$$
		n'=n-s,\qquad {\text{\rm  where}}\qquad  s=g^m+g^{m-1}. 
		$$ 
		If
	$\delta'_{m-1}\ne 0$ and $\delta_m'\ne 0$, we keep $n'$. Otherwise we consider the integer $n'=n-2s$. It is easy to check that  one of $n'=n-s$ or $n'=n-2s$ satisfies that 	$\delta'_{m-1}\ne 0$ and $\delta_m'\ne 0$.
				
		We distinguish two cases:
	\begin{itemize}
		\item[i)] The associated palindrome $p_1'$ of $n'$ has also $2m$ digits (this is the typical situation).
		
		We apply Algorithms  II or IV according to the type of $n'$. Then $n'=p_1'+p_2'+p_3'$ and so
	$$n=n'+ks=(p_1'+ks)+p_2'+p_3',\qquad k\in \{1,2\}.$$
	Notice that $p_1'+ks$ for $k\in \{1,2\}$ is also a palindrome because we are adding 1 or 2 to the two central digits of $p_1'$. Note that if we have applied Algorithm II, then the central digits are $x_m'$ and $x_m'$, which are $0$ or $1$ for $m\ge 3$. Note also that if we have applied Algorithm IV, then the central digits are $x_{m-1}'$ and $x_{m-1}'$, which are $0$ or $1$ for $m\ge 4$. Hence, in all the cases the value of the two central digits is at most $3$, which are legal digits for $g\ge 5$ (indeed, even for $g\ge 4$).
	
	\smallskip
	
	\item[ii)] The associated palindrome $p_1'$ of $n'$ has $2m-1$ digits. 
	
	This is only possible if $n$ is of the form $n=104\dots $ and $n'=103\dots $. In this special situation, we consider $n'$ as of type B1 or B2 and apply  the Algorithm IV  to $n'$ (instead of Algorithm I). Notice that the configuration of the starting point in B1 and B2 is also valid when $\delta_{l-3}=3$. Then the palindrome $p_1'$ we get in this way has $2m$ digits and, as above, we have
	$$n=n'+ks=(p_1'+ks)+p_2'+p_3',\qquad k\in \{1,2\}.$$
	\end{itemize}
		
		{\bf Example:} We finish with one example which shows how to apply Algorithms IV and V. Let  $n$ be the positive integer giving the first 20 digits of the Fibonacci factorial constant
		$$
		F=\prod_{k\ge 1} \left(1-a^k\right),\quad a=-\frac{1}{\phi^2}\quad {\text{\rm and}}\quad \phi=\frac{1+{\sqrt{5}}}{2}.
		$$
		Then				
$$
n=12267420107203532444.
$$
The number $n$ is a special number because it has an even number of digits, $20$, $m=10$ and $\delta_m=0$. Thus, we apply Algorithm V and consider $n'=n-s$, where $s=10^{10}+10^9.$ Note that
$n'=12267420096203532444$, which is a normal number because $\delta_m'\ne 0$ and $\delta_{m-1}'\ne 0.$

We observe that $n'$ is of type B.5, so we apply Algorithm IV to $n'$. The initial configuration is
$$\begin{array}{|ccccccccccc|ccccccccc|}\hline 1 & 2 & 2 &6 & 7& 4 & 2& 0&0 & 9& 6 &2& 0 & 3 & 5 &3&2&4&4& 4\\  \hline  \bf 1& \bf 1 & . & . & . & . &. & . & .&. & .& . &.& . & .& .&.&.& \bf1& \bf1\\ &  & \bf9 & . &. & . &  .& .&. & .& . &.& . & . & . &.&.&.& .&\bf 9\\ & &  &\bf4 & . & . & .& .&. & .& . &.& . & . & .&.&.&.& .& \bf 4\\ \hline
\end{array}$$
The temporary configuration is
$$\begin{array}{|ccccccccccc|ccccccccc|}\hline 1 & 2 & 2 &6 & 7& 4 & 2& 0&0 & 9& 6 &2& 0 & 3 & 5 &3&2&4&4& 4\\  \hline   1& 1 & 3 & 1 & 0 & 0 &0 & 0 &1&1& 1& 1 &0& 0 & 0& 0&1&3& 1& 1\\ &  & 9 & 1 &5 & 7 &  8& 5&0& 6& 1 &1& 6 & 0 & 5 &8&7&5& 1& 9\\ & &  &4 & 1 & 6 & 3& 4&9 & 2& 4 &9&4 & 2 & 9&4&3&6& 1&  4\\ \hline
\end{array}$$
Note that we need an adjustment because the digit in column 10 is not correct. The reason is that $x_{9}+c_9=2$. Looking at the central digits, we must follow the Adjustment Step IV.5.iii.a):
	$$\begin{array}{|c|ccccccccccc|ccccccccc|}\hline n'&1 & 2 & 2 &6 & 7& 4 & 2& 0&0 & 9& 6 &2& 0 & 3 & 5 &3&2&4&4& 4\\  \hline   p_1'&1& 1 & 3 & 1 & 0 & 0 &0 & 0 &1&\bf 0& \bf 0& 1 &0& 0 & 0& 0&1&3& 1& 1\\ p_2'& &  & 9 & 1 &5 & 7 &  8& 5&0& \bf 7& \bf 2 &\bf 2& \bf 7& 0 & 5 &8&7&5& 1& 9\\ p_3'& & &  &4 & 1 & 6 & 3& 4&9 & 2& \bf 3 &\bf 8&\bf 3 & 2 & 9&4&3&6& 1&  4\\ \hline
	\end{array}$$
	Finally, we add $s=10^{10}+10^9$ to $n'$ to obtain a representation of $n$ as a sum of three palindromes.
	$$\begin{array}{|c|ccccccccccc|ccccccccc|}\hline n&1 & 2 & 2 &6 & 7& 4 & 2& 0&1 & 0& 7 &2& 0 & 3 & 5 &3&2&4&4& 4\\  \hline   p_1&1& 1 & 3 & 1 & 0 & 0 &0 & 0 &1&\bf 1& \bf 1& 1 &0& 0 & 0& 0&1&3& 1& 1\\ p_2& &  & 9 & 1 &5 & 7 &  8& 5&0& 7&  2 & 2&  7& 0 & 5 &8&7&5& 1& 9\\ p_3& & &  &4 & 1 & 6 & 3& 4&9 & 2& 3 & 8&3 & 2 & 9&4&3&6& 1&  4\\ \hline
	\end{array}$$	
	\section{Small integers}\label{small}
	\begin{prop}\label{smallcases} All positive integers with less than seven digits are the sum of three palindromes in base $g\ge 5$.		
	\end{prop}
	\begin{proof}
		The proof is a consequence of the lemmas \ref{two}, \ref{three}, \ref{four}, \ref{five} and \ref{six}.
	\end{proof}
	\begin{lemma}\label{two} All  positive integers with two digits are the sum of two palindromes in base $g\ge 5$, except those of the form $n=(\delta+1)\delta,\ \ 1\le \delta\le g-2$, which are sum of three palindromes.
		\end{lemma}
		\begin{proof}
			Except for $n=10=(g-1)+1$, every positive integer $n$ with two digits in base $g\ge 5$ is of the form $n=\delta_1\delta_0$, and one of the following applies: 
			$$\qquad  \delta_1\le \delta_0 \qquad \qquad  \qquad \qquad \delta_1>\delta_0+1\qquad \qquad \qquad \qquad  \qquad \delta_1=\delta_0+1,~~~\delta_0\ge 1$$
			$\qquad \qquad \begin{array}{|cc|} \hline \delta_1&\delta_0\\  \hline \delta_1&\delta_1\\ &\delta_0-\delta_1\\ \hline
			\end{array}$ $\qquad \qquad \begin{array}{|cc|} \hline \delta_1&\delta_0\\  \hline \delta_1-1&\delta_1-1\\ &g+\delta_0-\delta_1+1\\ \hline
			\end{array}$ $\qquad \qquad \begin{array}{|cc|} \hline \delta_0+1&\delta_0\\  \hline \delta_0&\delta_0\\ &g-1\\ &1\\ \hline
			\end{array}$	
		\end{proof}
	\begin{lemma}\label{three}
		All positive integers with three digits are the sum of two palindromes in base $g\ge 5$, except $n=201$ which is the sum of three palindromes.
	\end{lemma}
	\begin{proof}
		Let $n=\delta_2\delta_1\delta_0$. 		
		$$\qquad \delta_2\le \delta_0  \qquad \qquad \qquad \delta_2\ge \delta_0+1,\  \delta_1\ne 0 \qquad  \qquad \delta_2\ge \delta_0+1,\ \delta_1= 0,\ D(\delta_2-\delta_0-1)\ne 0$$
		$\begin{array}{|ccc|} \hline \delta_2&\delta_1&\delta_0\\  \hline \delta_2&\delta_1&\delta_2\\ &&\delta_0-\delta_2\\ \hline
		\end{array}$ $\qquad\begin{array}{|ccc|} \hline \delta_2&\delta_1&\delta_0\\  \hline \delta_2&\delta_1-1&\delta_2\\ &&g+\delta_0-\delta_2\\ \hline
		\end{array}$ $\qquad\begin{array}{|ccc|} \hline \delta_2&\delta_1&\delta_0\\  \hline \delta_2-1&g-1&\delta_2-1\\ &&g+\delta_0-\delta_2+1\\ \hline
		\end{array}$
		
		\
	
				If $\delta_2\ge \delta_0+1$, $\delta_1= 0$, and $D(\delta_2- \delta_0-1)= 0$, we have that $\delta_0\equiv \delta_2-1\pmod g$ and we distinguish the following cases:
				$$ \delta_2\ge 3\qquad \qquad  \qquad \qquad \qquad  \delta_2= 2 \qquad \qquad \qquad \qquad \qquad  \delta_2= 1   $$
				$\qquad \begin{array}{|ccc|} \hline \delta_2&0&\delta_2-1\\  \hline \delta_2-2&g-1&\delta_2-2\\1 &1&1\\ \hline
				\end{array}$ $\qquad \qquad  \begin{array}{|ccc|} \hline 2&0&1\\  \hline 1&0&1\\ &g-1&g-1\\&&1 \\\hline
				\end{array}$ $\qquad \qquad  \begin{array}{|ccc|} \hline1&0&0\\  \hline  &g-1&g-1\\&&1 \\\hline
				\end{array}$
				
	\end{proof}

	\begin{lemma}\label{four} All positive integers with four digits are the sum of three palindromes in base $g\ge 5$.
			\end{lemma}
					\begin{proof}
				Let $n=\delta_3\delta_2\delta_1\delta_0.$ 
				
				\begin{itemize}

					\item[i)] $n\ge \delta_300\delta_3$, and $n$ is not of the form 
					$n=\delta_300\delta_3+m$ with $m=201$, or $m=(\delta+1)\delta$ with $\delta\ge 1$.  Then $n-\delta_300\delta_3$ is the  sum of two palindromes $p_1,p_2$ and 
					$$n=\delta_300\delta_3+p_1+p_2.$$
			\item[ii)] $n=\delta_300\delta_3+201$.
			$$\qquad \qquad \delta_3\ne 1,g-1\qquad \qquad \qquad \qquad  \qquad \delta_3=1 \qquad \qquad \qquad \qquad \delta_3=g-1$$
			$\begin{array}{|cccc|} \hline \delta_3&2&0&\delta_3+1\\  \hline \delta_3-1&g-1&g-1&\delta_3-1\\ &2&1&2 \\\hline
			\end{array}$ $\qquad \begin{array}{|cccc|} \hline 1&2&0&2\\  \hline 1&1&1&1\\ &&g-2&g-2\\ &&&3 \\\hline
			\end{array}$ $\qquad \begin{array}{|cccc|} \hline g-1&2&1&0\\  \hline g-1&1&1&g-1\\ &&g-2&g-2\\ &&&3 \\\hline
			\end{array}$

							\item[iii)] $n=\delta_300\delta_3+(\delta+1)\delta,
								\ 1\le \delta\le g-2:$ 
								
								\begin{itemize}
								 \item[a)]$\delta_3+\delta=\delta_0,\ $
								 
								 $$\delta_3\ne 1\qquad \qquad \qquad  \qquad \qquad \delta_3= 1$$
								 $\begin{array}{|cccc|} \hline \delta_3&0&\delta+1&\delta_0\\  \hline \delta_3-1&g-2&g-2&\delta_3-1\\ &1&3&1\\ &&\delta&\delta \\\hline
								 \end{array}$ $\qquad \begin{array}{|cccc|} \hline 1&0&\delta+1&\delta+1\\  \hline &g-1&g-1&g-1\\ &&\delta+1&\delta+1\\ &&&1 \\\hline
								 \end{array}$

							\item[b)]  $\delta_3+\delta= g+\delta_0$ with $0\le \delta_0\le g-1$:
							$\qquad \begin{array}{|cccc|} \hline \delta_3&0&\delta+2&\delta_0\\  \hline \delta_3-1&g-2&g-2&\delta_3-1\\ &1&3&1\\ &&\delta&\delta \\\hline
							\end{array}$	
						\end{itemize}	
						\item[iv) ]$n=\delta_300\delta_0$, $\delta_0\le \delta_3-1$ and $\delta_3\ne 1$. Then: 
						$\begin{array}{|cccc|} \hline \delta_3&0&0&\delta_0\\  \hline \delta_3-1&g-1&g-1&\delta_3-1\\ &&&g+\delta_0-\delta_3\\&&&1 \\\hline
						\end{array}$
						\item[v)] $n=1000$. Then:
						$\qquad \begin{array}{|cccc|} \hline 1&0&0&0\\  \hline &g-1&g-1&g-1\\ &&&1 \\\hline
						\end{array}$	
					\end{itemize}
				
			\end{proof}
				\begin{lemma}\label{five} All positive integers with five digits  are the sum of three palindromes		in base $g\ge 5$.		
				\end{lemma}
				\begin{proof} If $\delta_4\ne 1$, then $n$ is of type A and we apply Algorithm I, which works for $m=2$.
				
				Thus, we assume that $\delta_4=1$.			
				Let $n=1\delta_3\delta_2\delta_1\delta_0.$ 
				
				\begin{itemize}				
					\item[i)] $n\ge 1\delta_30\delta_31$ and $n$ is not of the form 
					$n=1\delta_30\delta_31+m$ with $m=201$, or $m=(\delta+1)\delta$ with $\delta\ge 1$.  By Propositions \ref{two} and \ref{three}, $n-1\delta_30\delta_31$ is the sum of 
					two palindromes $p_1,p_2$ and then
					$$n=1\delta_30\delta_31+p_1+p_2.  $$
						\item[ii)] $n=1\delta_30\delta_31+201$:
											$\begin{array}{|ccccc|} \hline 1&\delta_3&2&\delta_3&2\\  \hline 1&\delta_3&1&\delta_3&1\\ &&1&0&1 \\\hline
						\end{array}$
						
						\
								
								\item[iii)] $n=1\delta_30\delta_31+(\delta+1)\delta,
								\quad 1\le \delta\le g-2,\ \delta_3\ne 0:$ 
								
								\begin{itemize}
									\item[a)]$\delta+1+\delta_3\le g-1 :$
									
									$$\begin{array}{|ccccc|} \hline 1&\delta_3&0&\delta_3+\delta+1&\delta+1\\  \hline 1&\delta_3-1&1&\delta_3-1&1\\& &g-1&\delta+1&g-1\\ &&&&\delta +1\\ \hline
									\end{array}$$

									\item[b)]  $\delta_3+1+\delta= g+\delta_1$ with $0\le \delta_1\le g-1$:										
								$$\begin{array}{|ccccc|} \hline 1&\delta_3&1&\delta_1&\delta+1\\  \hline 1&\delta_3-1&1&\delta_3-1&1\\& &g-1&\delta+1&g-1\\ &&&&\delta +1\\ \hline
										\end{array}$$
								\end{itemize}
								
									\item[iv)] $n=1\delta_30\delta_31+(\delta+1)\delta,
									\quad 1\le \delta\le g-2,\ \delta_3= 0:$ 
										$$\begin{array}{|ccccc|} \hline 1&0&0&\delta+1&\delta+1\\  \hline &g-1&g-1&g-1&g-1\\& &&\delta+1&\delta+1\\ &&&&1\\ \hline
										\end{array}$$

					\item[v) ]$n\le 1\delta_30\delta_30$ and $\delta_3=0$. Then:
					$\begin{array}{|ccccc|} \hline 1&0&0&0&0\\  \hline &g-1&g-1&g-1&g-1\\&&&&1 \\\hline
					\end{array}$
					\item[vi)] $n\le 1\delta_30\delta_30$ and $\delta_3\ne 0$ with $n=1(\delta_3-1)(g-1)(\delta_3-1)1+m$ with $m\ne 201$ and $m\ne (\delta+1)\delta,\ 1\le \delta\le g-2$. Propositions \ref{two} and \ref{three} imply that $m$ is the sum of two palindromes $p_1,p_2$ and then
						$$n=1(\delta_3-1)(g-1)(\delta_3-1)1+p_1+p_2.   $$
						The remaining $n$ are of the form $n=1\delta_3\delta_2\delta_1\delta_0$, with $n\le 1\delta_30\delta_30$ (so, $\delta_2=0$) and also of the form $n=1(\delta_3-1)(g-1)(\delta_3-1)1+m$,
						where $m=201$ or $(\delta+1)\delta$. Then $m\ne 201$  (otherwise $n=1\delta_31(\delta_3-1)2$, so $\delta_2=1$), and $m\ne (\delta+1)\delta$ with $\delta_3+\delta\le g-1$,
						since otherwise $n=1(\delta_3-1)(g-1)(\delta_3+\delta)(\delta+1)$ (so the second leading digit of $n$ is $\delta_3-1$ not $\delta_3)$. Thus, the only possibility left is  the following
					    
					    \item[vii)] $n=1(\delta_3-1)(g-1)(\delta_3-1)1+(\delta+1)\delta,
								\ \delta_3\ne 0,\ \delta_3+\delta=g+\delta_1,\ 0\le \delta_1\le g-1:$ 
								$$\begin{array}{|ccccc|} \hline 1&\delta_3 & 0 &\delta_1&\delta+1\\  \hline 1&\delta_3-1&g-2&\delta_3-1&1\\& &1&\delta+1&1\\ &&&&\delta-1\\ \hline
								\end{array}$$
				\end{itemize}
			\end{proof}
				\begin{lemma}\label{six} All positive integers with six digits  are the sum of three palindromes in base $g\ge 5$.
				\end{lemma}
			\begin{proof} First, we consider the case $\delta_5\ne 1$.
				
		We apply Algorithm II for $m=3$ with some exceptions. Note that Algorithm II was applied to { \it normal numbers}. It was only used in the Adjustment Step II.2.ii.c), where we assumed that $\delta_2\ne 0$ and then  that $z_2\ne 0$ in that step. Thus, to apply Algorithm II when $n$ is not a {\it normal number}, we have to account also for the possibility $z_2=0$ in the Step II.2.ii.c). This is the temporary configuration in Step II.2.ii.c) ($c_2=0,\ y_3=y_2=0$) with $z_2=0$. 
			$$\begin{array}{|ccc|ccc|} \hline\delta_{5} & \delta_{4}&\delta_{3}&\delta_2&\delta_1&\delta_0\\  \hline x_1& x_2 &0&0&x_2&x_1\\   &y_1& 0& 0&0&y_1\\   & &z_1 & 0&0&z_1
			\\ \hline\end{array}$$
		If $x_2\ne 0$, then the adjustment step is the following:
			$$\begin{array}{|ccc|ccc|} \hline\delta_{5} & \delta_{4}&\delta_{3}&\delta_2&\delta_1&\delta_0\\  \hline x_1& x_2 &0&0&x_2&x_1\\   &y_1& 0& 0&0&y_1\\   & &z_1 & 0&0&z_1
			\\ \hline\end{array}\quad \longrightarrow \quad   \begin{array}{|ccc|ccc|} \hline\delta_{5} & \delta_{4}&\delta_{3}&\delta_2&\delta_1&\delta_0\\  \hline x_1& x_2-1 &g-1&g-1&x_2-1&x_1\\   &y_1& 1& 1&1&y_1\\   & &z_1 & 0&0&z_1
			\\ \hline\end{array}$$
			
		If $x_2=0$, we distinguish several cases:
			\begin{itemize}					
					\item[i)] $x_1=1$. It follows that $\delta_5=1$ (which is not allowed), unless $y_1=z_1= g-1$.  The adjustment step is the following:
					$$\begin{array}{|ccc|ccc|} \hline\delta_{5} & \delta_{4}&\delta_{3}&\delta_2&\delta_1&\delta_0\\  \hline 1& 0 &0&0&0&1\\   &g-1& 0& 0&0&g-1\\   & &g-1 & 0&0&g-1
					\\ \hline\end{array}\quad \longrightarrow \quad   \begin{array}{|ccc|ccc|} \hline\delta_{5} & \delta_{4}&\delta_{3}&\delta_2&\delta_1&\delta_0\\  \hline 2&0 &0&0&0&2\\   && &&1&1\\   & & &&&g-4
					\\ \hline\end{array}$$
						\item[ii)] $x_1\ne 1,\ y_1\ne g-1$.  The adjustment step is the following:
						$$\begin{array}{|ccc|ccc|} \hline\delta_{5} & \delta_{4}&\delta_{3}&\delta_2&\delta_1&\delta_0\\  \hline x_1& 0 &0&0&0&x_1\\   &y_1& 0& 0&0&y_1\\   & &z_1 & 0&0&z_1
						\\ \hline\end{array}\quad \longrightarrow \quad   \begin{array}{|ccc|ccc|} \hline\delta_{5} & \delta_{4}&\delta_{3}&\delta_2&\delta_1&\delta_0\\  \hline x_1-1&g-1 &0&0&g-1&x_1-1\\   &y_1+1& 0&g-2&0&y_1+1\\   & &z_1 &1&1&z_1
						\\ \hline\end{array}$$
							
							Let us look at the remaining possibilities. We have $x_1\ne 1$ and $y_1=g-1$. If $z_1\ne g-1$, then the temporary configuration is 
							$$\begin{array}{|ccc|ccc|} \hline\delta_{5} & \delta_{4}&\delta_{3}&\delta_2&\delta_1&\delta_0\\  \hline x_1& 0 &0&0&0&x_1\\   &y_1& 0& 0&0&y_1\\   & &z_1 & 0&0&z_1
							\\ \hline\end{array}
							$$
							and we see that there is no carry in the fourth column. That is $c_4=0$, so $y_1=\delta_4$, which contradicts the initial configurations given at A.1--A.4. Thus, $y_1=z_1=g-1$, and we distinguish two additional possibilities: 
							\medskip
							
								\item[iii)] $x_1\ne g-1,\ z_1=y_1= g-1$.  The adjustment step is the following:
								$$\begin{array}{|ccc|ccc|} \hline\delta_{5} & \delta_{4}&\delta_{3}&\delta_2&\delta_1&\delta_0\\  \hline x_1& 0 &0&0&0&x_1\\   &g-1& 0& 0&0&g-1\\   & &g-1& 0&0&g-1
								\\ \hline\end{array}\quad \longrightarrow \quad   \begin{array}{|ccc|ccc|} \hline\delta_{5} & \delta_{4}&\delta_{3}&\delta_2&\delta_1&\delta_0\\  \hline x_1+1&0 &0&0&0&x_1+1\\   && &&1&1\\   & & &&&g-4
								\\ \hline\end{array}$$
									\item[iv)] $ x_1=y_1=z_1= g-1$.  Note that in this case we have that 
									$$
									\delta_5\delta_4\delta_3\delta_2\delta_1\delta_0=(g-1)0000(g-1)+(g-1)000(g-1)+(g-1)00(g-1)+1000
									$$ 
									but we can check easily that this number has 7 digits.
				\end{itemize}
				
				\

		Secondly, we consider the case $\delta_5=1$.	
			
			\begin{itemize}
				\item[i)] $z_1=D(\delta_0-\delta_4+1)\ne 0$ and $D(\delta_0-\delta_4+2)\ne 0$. 
			
							$$\begin{array}{|cccccc|} \hline 1&\delta_{4} & \delta_{3}&\delta_2&\delta_1&\delta_0\\  \hline  & x_1& x_2& x_3& x_2&x_1 \\ &y_1&y_2&y_3&y_2&y_1\\ & & &z_1&z_2&z_1\\ \hline
				\end{array}$$
				
				We choose $x_1,y_1$ such that $1\le x_1,y_1\le g-1$ and $x_1+y_1=g+\delta_4-1$. This is possible because $2\le g+\delta_4-1\le 2g-2$.
				
				We choose $x_2,y_2$ such that $0\le x_2,y_2\le g-1$ and $x_2+y_2=g+\delta_3-1$. This is possible because $0\le g+\delta_4-1\le 2g-2$.				
				We also define $z_2=D(\delta_1-x_2-y_2-c_1)$.
				
				We choose $x_3,y_3$ such that $0\le x_3,y_3\le g-1$ and $x_3+y_3=g+\delta_2-c_2-z_1$. This is possible because, as $z_1\ne 0$, we have that $g+\delta_2-c_2-z_1\le 2g-2$, and since
			$D(\delta_0-\delta_4+2)\ne 0$, we have $z_1\ne g-1$ and therefore
				$$g+\delta_2-c_2-z_1\ge g+0-2-(g-2)=0.$$
				
				\
				
			\item[ii)]$D(\delta_0-\delta_4+2)= 0,\ \delta_2\ne 0.$
					$$\begin{array}{|cccccc|} \hline 1&\delta_{4} & \delta_{3}&\delta_2&\delta_1&\delta_0\\  \hline  & x_1& x_2& x_3& x_2&x_1 \\ &y_1&y_2&y_3&y_2&y_1\\ & & &z_1&z_2&z_1\\ \hline
					\end{array}$$
					
					We choose $x_1,y_1$ such that $1\le x_1,y_1\le g-1$ and $x_1+y_1=g+\delta_4-1$. Then $z_1=g-1$. We put $c_1=(x_1+y_1+z_1-\delta_0)/g=(2g+\delta_4-2-\delta_0)/g$. 
					
					We choose $x_2,y_2$ such that $0\le x_2,y_2\le g-1$ and $x_2+y_2=g+\delta_3-1$. We then put $z_2=D(\delta_1-x_2-y_2-c_1)$. 
					
					We choose $x_3,y_3$ such that $0\le x_3,y_3\le g-1$ and $x_3+y_3=g+\delta_2-c_2-z_1=(1+\delta_2)-c_2$. Here, $c_2=(x_2+y_2+z_2+c_1-\delta_1)/g$. 
					
					All such choices are possible by the same argument as in i) except that now we have to justify in a different way that $1+\delta_2-c_2\ge 0$, but this is clear because $\delta_2\ge 1$ and $c_2\le 2$.				
					
					\
						\item[iii)]$D(\delta_0-\delta_4+2)= 0,\ \delta_2=0.$
						\begin{itemize}
							\item[a)] $\delta_4=0$. Then $\delta_0=g-2$.	
							$$\begin{array}{|cccccc|} \hline 1&0 & \delta_{3}&0&\delta_1&g-2\\  \hline  & g-2& x_2& x_3& x_2&g-2 \\ &1&y_2&y_3&y_2&1\\ & &g-1&z_2&z_2&g-1\\ \hline
							\end{array}$$

							We choose $x_2,y_2$ such that $0\le x_2,y_2\le g-1$ and $x_2+y_2=\delta_3$. 
							
							We choose $x_3,y_3$ such that $0\le x_3,y_3\le g-1$ and $x_3+y_3=g-c_2-z_2$.
							
							Observe that $c_2=(x_2+y_2+z_2+c_1-\delta_1)/g\le (g-1+g-1+1)/g<2$. Thus, $c_2\ne 2$ and $g-c_2-z_2\ge g-1-(g-1)\ge 0$, therefore we can choose such $x_3$ and $y_3$.
							
							\
							
							\item[b)] $\delta_4=1$. Then $\delta_0=g-1$.	
							$$\begin{array}{|cccccc|} \hline 1&1 & \delta_{3}&0&\delta_1&g-1\\  \hline  & g-1& x_2& x_3& x_2&g-1 \\ &1&y_2&y_3&y_2&1\\ & &g-1&z_2&z_2&g-1\\ \hline
							\end{array}$$

						The choices for the $x_i$'s are identical to the ones from case a).
							
							\item[c)] $\delta_4=2$. Then $\delta_0=0$.	
							$$\begin{array}{|cccccc|} \hline 1&2 & \delta_{3}&0&\delta_1&0\\  \hline  & g-1& x_2& x_3& x_2&g-1 \\ &2&y_2&y_3&y_2&2\\ & &g-1&z_2&z_2&g-1\\ \hline
							\end{array}$$

							We choose $x_2,y_2$ such that $0\le x_2,y_2\le g-1$ and $x_2+y_2=\delta_3$. 
							
							We choose $x_3,y_3$ such that $0\le x_3,y_3\le g-1$ and $x_3+y_3=g-c_2-z_2$.
							
							If $c_2\ne 2$,  then we can make such a choice for $x_3$ and $y_3$.

							However, if $c_2=2$, then $x_2+y_2=z_2=g-1$ and $\delta_1=0$ and $\delta_3=g-1$. In this special case, we have:
							$$\begin{array}{|cccccc|} \hline 1&2 & g-1&0&0&0\\  \hline  1& 2& g-2& g-2& 2&1 \\ &&&1&g-3&1\\ & &&&&g-2\\ \hline
							\end{array}$$
							\item[d)] $\delta_4\ge 3$. Then $\delta_0=\delta_4-2\ge 1$. 
					         $$\begin{array}{|cccccc|} \hline 1& \delta_4& \delta_3&0&\delta_1& \delta_4-2\\  \hline 1 & 1-c_4 & 0 & 0 & 1-c_4 & 1 \\ &\delta_4-1 & D(\delta_3-1)& 2-c_2 & D(\delta_3-1)&\delta_4-1\\ & & &g-2&z&g-2\\ \hline
					\end{array}$$
					         Here, we first calculate
					         $$c_4=(D(\delta_3-1)+1-\delta_3)/g\in \{0,1\}.
					         $$
					         Next, $c_1=1$ and $z=D(\delta_1-\delta_3-1+c_4)$. Finally, 
					         $$
					         c_2=(2-c_4+D(\delta_3-1)+z-\delta_1)/g\in \{0,1,2\}.
					         $$ 
						\end{itemize}
						\medskip
						
					\item[iv)]$D(\delta_0-\delta_4+1)= 0,\ \delta_3\ne 0.$
					\begin{itemize}
					        \item[a)] $\delta_4\ne g-1:$
					$$\begin{array}{|cccccc|} \hline 1&\delta_{4} & \delta_{3}&\delta_2&\delta_1&\delta_0\\  \hline  & x_1& x_2& x_3& x_2&x_1 \\ &y_1&y_2&y_3&y_2&y_1\\ & & &z_1&z_2&z_1\\ \hline
					\end{array}$$
					We choose $x_1,y_1$ such that $1\le x_1,y_1\le g-1$ and $x_1+y_1=g+\delta_4$. This is possible because $\delta_4\le g-2$. On the other hand, $z_1=g-1$.
					
					We choose $x_2,y_2$ such that $0\le x_2,y_2\le g-1$ and $x_2+y_2=\delta_3-1$. 
					
					We choose $x_3,y_3$ such that $0\le x_3,y_3\le g-1$ and 
					$$x_3+y_3=g+\delta_2-c_2-z_1=1+\delta_2-c_2.
					$$
					This is possible because $c_2\le 1$. Indeed, 
					$$c_2=(x_2+y_2+z_2+c_1-\delta_1)/g\le (\delta_3-1+g-1+2)/g<2.$$
					         \item[b)] $\delta_4=g-1:$ 
					         $$\begin{array}{|cccccc|} \hline 1& g-1& \delta_{3}&\delta_2&\delta_1& g-2\\  \hline 1 & 3-c_3 & x-\mu & x-\mu & 3-c_3 & 1 \\ &g-4&y-c_2+\mu& D(\delta_2-x-1-c_1+\mu) &y-c_2+\mu&g-4\\ & & &1&D(\delta_1-3-y)+(c_2-\mu)+c_3&1\\ \hline
					\end{array}$$
					         \end{itemize}
					         In the above, $\mu\in \{0,1\}$. We choose $x,~y$ with $y\ge 1$ minimal such that $D(x+y)=\delta_3$ and $D(\delta_1-3-y)\not\in \{g-2,g-1\}$. Since the last condition forbids at most $2$ values for $y$, it follows that $y\in \{1,2,3\}$. 
					         Then 
					         $$c_1=(3+y+D(\delta_1-3-y)-\delta_1)/g\le (6+g-1-\delta_1)/g.
					         $$
					         The last expression is $<2$ if $g\ge 6$. For $g=5$, we can have $c_1=2$ only if $y=3,~\delta_1=0$, but then $y$ should have been chosen to be $1$, a contradiction.   
					         Thus, $c_1\in \{0,1\}$.  Next, we try $\mu=0$ and compute
					         $$
					         c_2=(x+D(\delta_2-x-1-c_1+\mu)+c_1+1-\delta_2)/g.
					         $$ 
					         If $c_2\in \{0,1\}$, we are all set. Otherwise, $c_2=2$, so $c_1=1$, $x=g-1,~\delta_2=0$. We then take $\mu=1$, getting $c_2=1$. Finally,
					         $$c_3=(x+(y-c_2)+c_2-\delta_3)/g\le (g-1+3+1)/g<2,$$
					         so $c_3\in \{0,1\}$.
					         
					         \medskip
					         
					\item[v)] 	$D(\delta_0-\delta_4+1)= 0,\ \delta_3=0$.
					
					\begin{itemize}
						\item[a)] $\delta_4=0$. Then $\delta_0=g-1$. 
						
						If $\delta_2\ne 0$, then
						$n-100001=\delta_2\delta_1(g-2)$  is a sum of two palindromes. 
						
						If $\delta_2=0$ and $\delta_1\ne 0, g-1$, then $n-100001=(\delta_1-1)(g-1)$ is also a sum of two palindromes.
						
							If $\delta_2=0$ and $\delta_1=0$, then 
							$$\begin{array}{|cccccc|} \hline 1&0 & 0&0&0&g-1\\  \hline 1 & 0&0& 0& 0&1\\ &&&&&g-2\\ \hline
							\end{array}$$
					If $\delta_2=0$ and $\delta_1=g-1$, then
					$$\begin{array}{|cccccc|} \hline 1&0 & 0&0&g-1&g-1\\  \hline  & g-1&0& 1& 0&g-1\\ &&g-1&g-2&g-2&g-1\\ & &&1&0&1\\ \hline
					\end{array}$$
						\item[b)] $\delta_4=1$. Then $\delta_0=0$. 
						
						If $\delta_2\ge 2$ or if $\delta_2=1$ and $\delta_1\ne 0,1$ then
						$n-110011$ has three digits, its last digit is $g-1$, therefore it can be written as a sum of two palindromes.
						
					If $\delta_2=1$ and $\delta_1=0$, then
						$$\begin{array}{|cccccc|} \hline 1&1 & 0&1&0&0\\  \hline  1& 0&g-1& g-1&0&1\\ &&&1&g-1&1\\ & &&&&g-2\\ \hline
						\end{array}$$

						If $\delta_2=1$ and $\delta_1=1$, then						
							$$\begin{array}{|cccccc|} \hline 1&1 & 0&1&1&0\\  \hline  1& 1&0& 0&1&1\\ &&&&g-1&g-1\\ \hline
					\end{array}$$
						
						If $\delta_2=0$ and $\delta_1\ge 2$, then
						$$\begin{array}{|cccccc|} \hline 1&1 & 0&0&\delta_1&0\\  \hline  1& 1&0& 0&1&1\\ &&&&\delta_1-2&\delta_1-2\\ & &&&&g-\delta_1+1\\ \hline
						\end{array}$$
						
						If $\delta_2=0$ and $\delta_1=1$, then
						$$\begin{array}{|cccccc|} \hline 1&1 & 0&0&1&0\\  \hline  1& 0&0& 0&0&1\\ &1&0&0&0&1\\  &&&&&g-2 \\\hline
						\end{array}$$
						 If $\delta_2=0$ and $\delta_1=0$ then
							$$\begin{array}{|cccccc|} \hline 1&1 & 0&0&0&0\\  \hline  1& 0&0& 0&0&1\\ &&g-1&g-1&g-1&g-1\\  \hline
							\end{array}$$
						
						\item[c)] $\delta_4=2$. Then $\delta_0=1$. 
						
							If $\delta_2\ge 2$ or if $\delta_2=1$ and $\delta_1\ne 0,1$, then
							$n-120021$ has three digits, its last digit is $g-1$, therefore can be written as a sum of two palindromes.
							
								If $\delta_2=1$ and $\delta_1=0$, then						$$\begin{array}{|cccccc|} \hline 1&2 & 0&1&0&1\\  \hline  1& 1&g-1&g-1&1&1\\ &&&1&g-2&1\\ & &&&&g-1\\ \hline
								\end{array}$$

								If $\delta_2=1$ and $\delta_1=1$, then						
							$$\begin{array}{|cccccc|} \hline 1&2 & 0&1&1&1\\  \hline  1& 1&g-1&g-1&1&1\\ &&&1&g-1&1\\ & &&&&g-1\\ \hline
							\end{array}$$

								If $\delta_2=0$ and $\delta_1\ge 3$, then
								$$\begin{array}{|cccccc|} \hline 1&2 & 0&0&\delta_1&1\\  \hline  1& 2&0& 0&2&1\\ &&&&\delta_1-3&\delta_1-3\\ & &&&&g-\delta_1+3\\ \hline
								\end{array}$$
								(In the above, if $\delta_1=3$, then the second palindrome is missing and the last is $g$ which is the sum $(g-1)+1$).
									
									If $\delta_2=0$ and $\delta_1=2$, then
									$$\begin{array}{|cccccc|} \hline 1&2 & 0&0&2&1\\  \hline  1& 1&g-1& g-1&1&1\\ &&&1&0&1\\  &&&&&g-1 \\\hline
									\end{array}$$
									
								If $\delta_2=0$ and $\delta_1=1$, then
								$$\begin{array}{|cccccc|} \hline 1&2 & 0&0&1&1\\  \hline  1& 0&0& 0&0&1\\ &2&0&0&0&2\\  &&&&&g-2 \\\hline
								\end{array}$$
								If $\delta_2=0$ and $\delta_1=0$, then
								$$\begin{array}{|cccccc|} \hline 1&2 & 0&0&0&1\\  \hline  1& 1&g-1& g-1&1&1\\ &&&&g-2&g-2\\ &&&&&2\\ \hline
								\end{array}$$
								
						\item[d)] $\delta_4=3$. Then $\delta_0=2$ and
						$$\begin{array}{|cccccc|} \hline 1& 3& 0&\delta_2&\delta_1& 2\\  \hline 1 & 0 & g-y-1-c_1 & g-y-1-c_1 & 0& 1 \\ & 2&y-c_2+1+c_1& D(\delta_2+y+2) &y-c_2+1+c_1&2\\ & & &g-1&D(\delta_1-1-y)+(c_2-1)-c_1&g-1\\ \hline
					\end{array}$$
					         We choose $y$ with $y\ge 1$ minimal such that $D(\delta_1-1-y)\not\in \{0,g-1\}$. Since the last condition forbids at most $3$ values for $y$, it follows that $y\in \{1,2,3\}$. 
					         Then 
					         $$c_1=(2+y+D(\delta_1-1-y)-\delta_1)/g\le (5+g-1)/g<2.
					         $$	         
					         Thus, $c_1\in \{0,1\}$.  Next, 
					         $$
					         c_2=(g-y-1+D(\delta_2+y+2)+g-1-\delta_2)/g.
					         $$ 
					         Clearly, $(g-y-1)+g-1\ge 2g-5\ge g$. Thus, $c_2\in \{1,2\}$. Thus, $c_2-1\in \{0,1\}$. Thus, $y-(c_2-1)+c_1\in [0,4]$ so it is a digit for all $g\ge 5$. Also, $g-y-1-c_1\ge g-5\ge 0$.
					         
					         \medskip
					         
					         \item[e)] $\delta_4\ge 4$. Then $\delta_0=\delta_4-1$.
					         $$\begin{array}{|cccccc|} \hline 1& \delta_4& 0&\delta_2&\delta_1& \delta_4-1\\  \hline 1 & 2 & g-y-c_1 & g-y-c_1 & 2 & 1 \\ &\delta_4-3 & y-c_2+c_1& D(\delta_2+y-1) &y-c_2+c_1&\delta_4-3\\ & & &1&D(\delta_1-2-y)+c_2-c_1&1\\ \hline
					\end{array}$$
					         We choose $y\ge 1$ minimal such that $D(\delta_1-1-y)\not\in \{0,g-1\}$. Since the last condition forbids at most $2$ values for $y$, it follows that $y\in \{1,2,3\}$. 
					         Then 
					         $$c_1=(1+y+D(\delta_1-1-y)-\delta_1)/g\le (4+g-1-\delta_1)/g<2.
					         $$
					         Thus, $c_1\in \{0,1\}$.  Next,
					         $$
					         c_2=(g-y+1+D(\delta_2+y-1)-\delta_2)/g\le (g+g-1)/g<2.
					         $$ 
					         Thus, $c_2\in \{0,1\}$. Thus, $y-c_2+c_1\in [0,4]$ and $D(\delta_1-2-y)+c_2-c_1\in \{0,\ldots,g-1\}$.

						\end{itemize}

				\end{itemize}	
		\end{proof}		
\section{The proofs of Theorems \ref{pp} and \ref{ppu}}

\subsection{Proof of Theorem \ref{pp}}To get the lower bound we argue in the following way. Let $P_l$ be the set of palindromes with  $l$ base $g$  digits. Its cardinality is bounded by $g^{(l+1)/2}$. Let $X$ be large and $l$ be that positive integer such that $2g^l\le X<2g^{l+1}$.  
It is clear that for all $r\ge 1$, $|P_l+P_{l-r}|$ is a lower bound for the number of positive integers less than or equal to $X$ which are a sum of two base $g$ palindromes. 
We use the relation
$$
|P_l||P_{l-r}|=\sum_{n\in P_l+P_{l-r}}r(n)\le |P_l+P_{l-r}|\max_{n\in P_l+P_{l-r}}r(n).
$$
Consider the representations of  $n$ of the form $n=x+y$ with  $x\in P_l$ and $y\in P_{l-r}$. Assume that $l=2mr+t$, with $0\le t\le 2r-1$. 

If 
$$x=x_1x_2\dots x_2x_1\quad {\text{\rm and}}\quad y=y_1y_2\dots y_2y_1
$$ 
are the base $g$ representations of $x$ and $y$, then we group the digits in blocks of length $r$ from the left to the right and we get left over with a middle block  of length $t$:
$$
x=\underline{x_1\dots x_r}\ \cdots  \underline{x_{2r(m-1)+1}\dots x_{2rm}}\  \underline{x_{2rm+t}\dots x_{2rm+1}}\ \underline{x_{2rm}\dots x_{2r(m-1)+1}}\cdots  \underline{x_r\dots x_1}.
$$
If $X=x_1\dots x_r$, we define $f(X):=x_r\dots x_1$. With this notation, $x$ and $y$ are represented as $X_i,Y_i, f(X_i),\ f(Y_i),\ \Delta_3$ of length $r$, while $\Delta_1,\Delta_2$ have length $t$:
$$\begin{array}{ccccccccccccccc}x=&X_1 & \cdots &   \cdots   &\cdots  & X_{m} &\Delta_1 & f(X_{m}) &\cdots  & \cdots    &\cdots   & f(X_1)\\  y=&& f(Y_1) & \cdots &  \cdots & f(Y_{m-1}) & \Delta_2 &\Delta_3  & Y_{m-1} & \cdots & \cdots  & Y_1
\end{array}
$$
When we sum $x$ and $y$, digit by digit, in every column we could get a carry or not. Let $t_i$ for $i=1,\dots, 2m$ be the carries  in each column and let $\overline t=(t_1,\dots,t_{2m})$ be the vector of carries. We denote by $r_{\overline t}(n)$ 
the number of representations of  $n$ under the form $n=x+y$ with $x\in P_l,~y\in P_{l-r}$ with a carries vector $\overline t$. Clearly, 
$$
r(n)=\sum_{\overline t}r_{\overline t}(n).
$$
As in the case of $x$ and $y$, we write $n$ with the same length of the string of digits as $x$. 
$$\begin{array}{ccccccccccccccc}n=&\delta_{2m} & \cdots & \cdots  & \cdots   &\cdots  & \cdots &\cdots & \cdots &\cdots  & \cdots &\cdots   &\cdots   & \delta_0\\ \hline \\ x=&X_1 & \cdots & \cdots  & \cdots   &\cdots  &\cdots &\cdots  &\cdots &\cdots  & \cdots &\cdots   &\cdots   & f(X_1)\\  y=&& f(Y_1) & \cdots &  \cdots & \cdots & \cdots & \cdots &\cdots  & \cdots & \cdots & \cdots & \cdots  & Y_1
\end{array}$$
Let us see that $X_i,Y_i,\Delta_1,\Delta_2,\Delta_3$ are all determined by $\delta_i$ and by the vector $\overline t$.
	
In fact, $X_1$ is determined by $\delta_{2m}$ and $t_{2m}$. We then put $f(X_1)$, which in turn determines $Y_1$. If the carry in the first column does not coincide with $t_1$, then $r_{\overline t}(n)=0$. If it does, then we put $f(Y_1)$ in its appropriate position. 
We then determine $X_2$ using $\delta_{2m-1}$ and $t_{2m-1}$. Again if the carry in the second column does not correspond with $t_2$, then $r_{\overline t}(n)=0$; otherwise, we keep on determining $X_i,Y_i$ and $\Delta_3$. If one of these determinations is not compatible with the corresponding  $t_i$'s then $r_{\overline t}(n)=0$. In the last step, we have to determine what is $\Delta_1$. Since $\Delta_1$ is a palindrome itself and has length $t$, there are at most $g^r$ possibilities for it. Once we made up our mind about $\Delta_1$, the value of $\Delta_2$ is determined. So, $r_{\overline{t}}(n)\le g^r$ and therefore $r(n)\le 2^mg^r.$
	
Hence, 
\begin{eqnarray*}
|P_l+P_{l-r}| & \ge &  g^{l+1-r/2}2^{-m}g^{-r}\\
& \ge &  (X/2)g^{-3r/2}2^{-m}\\
& \ge & (X/2)g^{-3r/2}2^{-l/(2r)}\\
& \ge &  (X/2)g^{-\frac 12(3r+\frac{l\log g}{r\log 2})}.
\end{eqnarray*}
Taking $r=\lfloor \sqrt{l(\log g)/(3\log 2)}\rfloor$ and using the fact that $l\sim \log X/\log g$, we get
$$
|P_l+P_{l-r}|\gg Xg^{-\sqrt{3l\log g/\log 2}}\gg X e^{-c\sqrt{\log X}}.
$$
	\subsection{Proof of Theorem \ref{ppu}}

		 For $g\ge 3$, it is not hard to see that  the number
		\begin{equation}
		\label{eq:n}
		(g-1) (g-1) * * * \cdots * 0 (g-1)
		\end{equation}
		is not a sum of two base $g$ palindromes. Indeed, assume that the length of the above $n$ is $l\ge 4$ and that  $x=x_{l-1-r}\cdots x_0\ge y=y_{l-1-s}\cdots y_0$ are  base $g$ palindromes whose sum is the above $n$, where $r,~s$ are nonnegative integers. 
		Since $x_0+y_0\le 2g-2$ and the last digit of $n$ is $g-1$, there is no carry in the last position when summing $x$ and $y$ in base $g$, so $x_0+y_0=g-1$ with $1\le x_0,y_0\le g-2$. If both $r>0$ and $s>0$ (so, the lengths of both $x$ and $y$ are smaller than $l$), then $n=x+y$ which has length $l$ in base $g$ should start with $1$, 
		which is not the case. If $r=0$ but $s>0$, then $x_0=g-2$ and $y_0=1$. Since $y_{l-2}=1$ or $0$ according to whether $s=1$ or $s\ge 2$, respectively, 
		and since there is a carry in the position $l-2$ when adding $x$ with $y$, we conclude that $x_{l-2}=g-2$ or $g-1$. But then $g+1\ge x_{l-2}+y_{l-2}+1\ge g+(g-1)=2g-1$, where the last inequality follows from the fact that the digit in the position 
		$l-2$ of $n$ is $g-1$, and the above string of inequalities is impossible. Hence, $r=s=0$. Now looking at $x_1$ and $y_1$, we get that $x_1+y_1=0$ or $g$. Looking now at the left, we conclude that $x_{l-2}+y_{l-2}=x_1+y_1=0$ or $g$, so in the position $l-2$ of the digits of $n$ we 
		should have either the digit $0$ or $1$ according to whether there is no carry coming from the sum of digits of $x$ and $y$ from the position $l-3$, or if there is one such carry, respectively, and both these numbers are smaller  than the corresponding digit $g-1$ of $n$, which is the final contradiction. 
		
{\bf Note.} All algorithms in the paper have been implemented in a Python program which generated a representation of $n$ as  a sum of three base $g$ palindromes for all $n$ and $g$ in the ranges $g+\ell\le 17$ and $g\in \{5,6,7,8,9,10\}$. The program is available from the third author upon request.		
		
\section*{Acknowledgements}

The authors thank the referees for comments which improved the quality of this paper. Work on this paper started during a visit of Florian Luca to the Mathematics
Department of the UAM in November 2015 and finished during a visit by the same author to the Max Planck Institute for Mathematics in Bonn during the period January to July of 2017. This author thanks these Institutions
for their hospitality and support. Javier Cilleruelo has been supported by MINECO project MTM2014-56350-P
and by ICMAT Severo Ochoa project SEV-2015-0554 (MINECO). Florian Luca
was supported in part by a start-up grant from Wits University and by an NRF
A-rated researcher grant.

\end{document}